\documentclass{article}
\usepackage[utf8]{inputenc}

\title{Global stabilization of the cubic defocusing nonlinear Schrödinger equation on the torus}
\author{Kévin Le Balc'h, Jérémy Martin}

\usepackage[top=3cm, bottom=2cm, left=3cm, right=3cm]{geometry}	

\usepackage{enumitem}
\usepackage{array}											

\usepackage{amsfonts}
\usepackage{amsmath}
\usepackage{amsthm}
\usepackage{amssymb}
\usepackage{mathtools}

\usepackage{multicol}

\usepackage{bbm}

\usepackage{stmaryrd}

\usepackage[scr]{rsfso}

\usepackage{comment}

\usepackage[dvipsnames]{xcolor}

\usepackage{hyperref}

\hypersetup{	
  colorlinks=true,
  breaklinks=true,
  urlcolor= red,
  linkcolor= red,
  citecolor=ForestGreen
}

\numberwithin{equation}{section}

\newtheorem{theorem}{Theorem}[section]

\newtheorem{lemma}[theorem]{Lemma}
\newtheorem{proposition}[theorem]{Proposition}

\newtheorem{corollary}[theorem]{Corollary}

\numberwithin{equation}{section}

\newcommand{\norme}[1]{\left\lVert#1\right\rVert}

\newcommand{\ensemblenombre}[1]{\mathbb{#1}}

\newcommand{\Z}{\ensemblenombre{Z}}

\newcommand{\R}{} 
\renewcommand{\R}{\ensemblenombre{R}}
\newcommand{\C}{\ensemblenombre{C}}

\newcommand{\T}{\ensemblenombre{T}}

\newcommand{\rr}{\mathbb{R}}
\newcommand{\nn}{\mathbb{N}}

\newcommand{\Ree}{\text{Re}}
\newcommand{\Ima}{\text{Im}}

\newcommand\bna{\begin{eqnarray*}}
\newcommand\ena{\end{eqnarray*}}

\newcommand\bnan{\begin{eqnarray}}
\newcommand\enan{\end{eqnarray}}

\begin{document}

\maketitle

\begin{abstract}
    In this article, we prove the (uniform) global exponential stabilization of the cubic defocusing Schrödinger equation on the torus $(\R/2 \pi \Z)^d$, for $d=1$, $2$ or $3$, with a linear damping localized in a subset of the torus satisfying some geometrical assumptions. In particular, this answers an open question of Dehman, Gérard, Lebeau from 2006. Our approach is based on three ingredients. First, we prove the well-posedness of the closed-loop system in Bourgain spaces. Secondly, we derive new Carleman estimates for the nonlinear equation by directly including the cubic term in the conjugated operator. Thirdly, by conjugating with energy estimates and Morawetz multipliers method, we then deduce quantitative observability estimates leading to the uniform exponential decay of the total energy of the system. As a corollary of the global stabilization result, we obtain an upper bound of the minimal time of the global null-controllability of the nonlinear equation by using a stabilization procedure and a local null-controllability result.
\end{abstract}
 \small
\tableofcontents
\normalsize

\section{Introduction}

\subsection{Control of Schrödinger equations on compact manifolds}

Let $(\mathcal M,g)$ be a compact smooth connected boundaryless Riemannian $d$-dimensional manifold, for $d \in \{1, 2, 3\}$ and $\Delta_g$ be the Laplace Beltrami operator on $\mathcal M$ associated to the metric $g$. Very quickly, we will restrict ourselves to $\mathcal M = \T^d = (\R/2 \pi \Z)^{d}$.

We are interested in the cubic defocusing nonlinear Schrödinger equation
\begin{equation}
	\label{eq:NonlinearSchrodinger_manifold}
		\left\{
			\begin{array}{ll}
				i  \partial_t u  = -\Delta_g u + |u|^{2} u & \text{ in }  (0,+\infty) \times \mathcal M, \\
				u(0, \cdot) = u_0 & \text{ in } \mathcal M.
			\end{array}
		\right.
\end{equation}
This equation arises naturally in nonlinear optics, as a model of wave propagation in fiber optics. The function $u(t,x) \in \C$ represents a wave and the nonlinear Schrödinger equation describes the propagation of the wave through a nonlinear medium. In this context, the metric $g$ can be interpreted as an inhomogeneity of the optical index.

Formally, in \eqref{eq:NonlinearSchrodinger_manifold}, two quantities are conserved. First, by multiplying by $\overline{u}$ the equation \eqref{eq:NonlinearSchrodinger_manifold} and by taking the imaginary part, we observe that the $L^2$-energy is conserved, i.e.
\begin{equation}
   \label{eq:conservL2}
    \frac{d}{dt} \left(\int_{\mathcal M} |u(t,x)|^2 dx \right) = 0\qquad \forall t \geq 0.
\end{equation}
Secondly, by multiplying by $\partial_t \overline{u}$ the equation \eqref{eq:NonlinearSchrodinger_manifold} and by taking the real part, we also observe that the nonlinear-energy (or $H^1$-energy) is conserved, i.e.
\begin{equation}
    \label{eq:conservEnergyH1}
  \frac{d}{dt} \left( \frac{1}{2}\int_{\mathcal M} |\nabla u(t,x)|^2 dx +   \frac{1}{4} \int_{\mathcal M} |u(t,x)|^{4}  dx \right) = 0 \qquad \forall t \geq 0.
\end{equation}

Concerning the (global) well-posedness of \eqref{eq:NonlinearSchrodinger_manifold} for initial data in $H^1(\mathcal M)$, in the $1$-dimensional case, this comes from energy estimates and Sobolev embeddings, see for instance \cite[Corollary 3.5.2]{Caz03}. However, this strategy fails in the $d$-dimensional case ($d \geq 2$), see for instance \cite[Corollary 3.5.2]{Caz03}. In order to obtain global existence results in $H^s(\mathcal M)$ for $s \geq 1$, one needs to use Strichartz-type estimates. For $\mathcal M = \T^d$, we have that \eqref{eq:NonlinearSchrodinger_manifold} is globally well-posed in $H^s(\T^d)$ for every $s \geq 1$, up to dimension $d=3$, see \cite{Bou93} and \cite[Chapter 5]{Bou99}. For $\mathcal M = S^d$, \eqref{eq:NonlinearSchrodinger_manifold} is globally well-posed in $H^s(\T^d)$ for every $s \geq 1$, up to dimension $d=3$, see \cite{BGT05a} for surfaces and \cite{BGT05b} for $d=3$.

The goal of the paper would be to analyse controllability and stabilization properties of \eqref{eq:NonlinearSchrodinger_manifold} with the help of a force $h$ localized in $\omega$, a nonempty open subset of $\mathcal M$, satisfying some geometrical assumptions, see below.\\

We first introduce the controlled linear Schrödinger equation
\begin{equation}
	\label{eq:linearSchrodinger_manifold}
		\left\{
			\begin{array}{ll}
				i  \partial_t u  = -\Delta_g u + h 1_{\omega} & \text{ in }  (0,+\infty) \times \mathcal M, \\
				u(0, \cdot) = u_0 & \text{ in } \mathcal M.
			\end{array}
		\right.
\end{equation}
In \eqref{eq:linearSchrodinger_manifold}, at time $t \in (0,+\infty)$, $u(t, \cdot) : \mathcal M \to \C$ is the state and $h(t, \cdot) : \omega \to \C$ is the control.

Controllability for the linear Schrödinger equation has been started to be strongly investigated in the 1990's. We recall the definitions of controllability and its dual notion, called observability.\\

Let $s \geq 0$ and $T>0$. 

The equation \eqref{eq:linearSchrodinger_manifold} is \textbf{exactly controllable} in $H^s(\mathcal M)$ at time $T>0$ if for every $u_0 \in H^s(\mathcal M)$ and $u_1 \in H^s(\mathcal M)$, there exists $h \in L^2(0,T;H^s(\mathcal M))$ such that the mild solution $u$ of \eqref{eq:linearSchrodinger_manifold} belongs to $C([0,T];H^s(\mathcal M))$ and satisfies $u(T, \cdot) = u_1$.


The linear Schrödinger equation is \textbf{observable} in $H^{-s}(\mathcal M)$ at time $T>0$ if there exists a constant $C = C(\mathcal M, \omega, T)>0$ such that for every $u_0 \in H^{-s}(\mathcal M)$, $\norme{u_0}_{H^{-s}(\mathcal M)}^2 \leq C \int_{0}^{T} \norme{e^{it \Delta}u_0 1_{\omega}}_{H^{-s}(\mathcal M)}^2 dt$.

The \textbf{Hilbert Uniqueness Method} (H.U.M.) relates the two previous notions, see for instance \cite[Theorem 2.42]{Cor07}. The controlled linear Schrödinger equation \eqref{eq:linearSchrodinger_manifold} is exactly controllable in $H^s(\mathcal M)$ at time $T>0$ if and only if the the linear Schrödinger equation is observable in $H^{-s}(\mathcal M)$ at time $T>0$.\\

In this direction, one of the most breakthrough result is from \cite{Leb92} that guarantees that the so-called Geometric Control Condition (GCC) for the wave equation equation is sufficient for the exact controllability of the Schrödinger equation in any time $T>0$. The proof of this result is based on microlocal analysis. The GCC can be, roughly, formulated as follows, the subdomain $\omega$ is said to
satisfy the GCC in time $T >0$ if and only if all rays of Geometric Optics that propagate inside the domain reach the control set $\omega$ in time less than $T$. A particular case of this result was proved previously by \cite{Mac94} by multiplier techniques. One can also see \cite{LTZ04} for the obtention of such results by using Carleman estimates. In \cite{Phu01}, the author establishes the connections between the heat, wave and Schrödinger equations through suitable integral transformations called Fourier-Bros-Iagonitzer (FBI) transformations. This allows him to get, for instance, estimates on the cost of approximate controllability for the Schrödinger equation when the GCC is not satisfied and also on the dependence of the size of the control with respect to the control time. For the sphere $S^d$, by using explicit quasimodes that concentrate on the equator, one can prove that the GCC is necessary and sufficient for the observability. On the other hand, there a number of results showing that, in some situations in which the GCC is not fulfilled in any time, one can still achieve very satisfactory results for the Schrödinger equation. For instance, any nonempty open subset of $\T^d$ is sufficient for the observability then for the exact controllability of the Schrödinger equation, see \cite{Jaf90}, \cite{KL05} by using Ingham's estimates or \cite{AM14} by using semi-classical measures. Note that it has been recently extended to any nontrivial measurable subset of $\T^2$ in \cite{BZ19} with the crucial use of dispersive properties of the Schrödinger equation. For the unit disk $\mathbb D$ with Dirichlet boundary conditions, explicit eigenfunctions concentrate near the boundary so one can prove that the observability holds if and only if $\omega$ contains a (small) part of the boundary $\partial \mathbb D$, see \cite{ALM16}. For a survey of these results up to 2002, one can read \cite{Zua03}.\\

For $\lambda \in \R^{*}$, the controlled cubic focusing ($\lambda <0$) or defocusing ($\lambda >0$) Schrödinger equation writes as follows
\begin{equation}
	\label{eq:NonlinearSchrodinger_manifoldControlled}
		\left\{
			\begin{array}{ll}
				i  \partial_t u  = -\Delta u + \lambda |u|^{2} u + h 1_{\omega} & \text{ in }  (0,+\infty) \times \mathcal M, \\
				u(0, \cdot) = u_0 & \text{ in } \mathcal M.
			\end{array}
		\right.
\end{equation}
In \eqref{eq:NonlinearSchrodinger_manifoldControlled}, at time $t \in (0,+\infty)$, $u(t, \cdot) : \mathcal M \to \C$ is the state and $h(t, \cdot) : \omega \to \C$ is the control.

Controllability and stabilization properties have been started to be investigated at the beginning of the $2000's$. These notions are usually split into local, semiglobal and global. We just give the definitions for the stabilization in $H^s(\mathcal M)$, for a feedback operator $P \in \mathcal{L}(H^s(\T^d))$, to illustrate the differences between them and because we will mainly focus on it after. Other precise definitions, in particular concerning controllability, can be found in references mentioned below.\\

Let $s \geq 0$ and assume that there exists a family of Hilbert spaces $(E_{T,s})_{T>0} \subset \mathcal C([0,T], H^s(\mathcal M))$ such that for all $T>0$, the equation \eqref{eq:NonlinearSchrodinger_manifoldControlled} posed on the time interval $[0,T]$ with $h=0$ and $u_0 \in H^s(\mathcal M)$ is well-posed on the space $E_{T,s}$.

The equation \eqref{eq:NonlinearSchrodinger_manifoldControlled} is \textbf{locally exponentially stabilizable} if there exists $P\in \mathcal{L}(H^s(\mathcal M))$ such that the equation \eqref{eq:NonlinearSchrodinger_manifoldControlled} with $h = Pu$ is well-posed on $E_{T,s}$ for all $T>0$ and there exist $\delta >0$, $C>0$ and $\gamma>0$ such that for every $u_0 \in H^s(\mathcal M)$ satisfying $\|u_0\|_{H^s(\mathcal M)} \leq \delta$, 
the solution $u$ of \eqref{eq:NonlinearSchrodinger_manifoldControlled} satisfies $\|u(t, \cdot)\|_{H^s(\mathcal M)} \leq C e^{- \gamma t}$ for every $t \geq 0$.

The equation \eqref{eq:NonlinearSchrodinger_manifoldControlled} is \textbf{semiglobally exponentially stabilizable} if there exists $P\in \mathcal{L}(H^s(\mathcal M))$ such that the equation \eqref{eq:NonlinearSchrodinger_manifoldControlled} with $h = Pu$ is well-posed on $E_{T,s}$ for all $T>0$ and for every $R>0$, there exist $C=C(R)>0$ and $\gamma=\gamma(R)>0$ such that for every $u_0 \in H^s(\mathcal M)$ satisfying $\|u_0\|_{H^s(\mathcal M)} \leq R$, the solution $u$ of \eqref{eq:NonlinearSchrodinger_manifoldControlled} satisfies $\|u(t, \cdot)\|_{H^s(\mathcal M)} \leq C e^{- \gamma t}$ for every $t \geq 0$.

The equation \eqref{eq:NonlinearSchrodinger_manifoldControlled} is  \textbf{globally exponentially stabilizable} if there exists $P\in \mathcal{L}(H^s(\mathcal M))$ such that the equation \eqref{eq:NonlinearSchrodinger_manifoldControlled} with $h = Pu$ is well-posed on $E_{T,s}$ for all $T>0$ and there exist $C>0$ and $\gamma>0$ such that for every $u_0 \in H^s(\mathcal M)$, the solution $u$ of  \eqref{eq:NonlinearSchrodinger_manifoldControlled} satisfies $\norme{u(t,\cdot)}_{H^s(\mathcal M)} \leq C e^{- \gamma t} \norme{u_0}_{H^s(\mathcal M)}$ for every  $t \geq 0$.\\





We consider $a \in C_c^{\infty}(\omega)$ such that $a(x) \geq a_0 > 0$ in $\widehat{\omega} \subset \subset \omega$ where $\omega$ is a nonempty open subset of $\mathcal M$. Local exact controllability in $H^1(\T)$ for \eqref{eq:NonlinearSchrodinger_manifoldControlled} has been first obtained in \cite{ILT03}. It has then been extended in $H^s(\T)$ in \cite{RZ09} for every $s \geq 0$ by using moments theory and Bourgain analysis for the treatment of the semilinearity seen as a small perturbation of the linear case. Note that the local stabilization with the feedback $h = - i a(x) u$ has also been obtained in \cite{RZ09}. This type of result has been generalized to any $d$-dimensional torus $\T^d$, $d \geq 2$, in \cite{RZ10}. In \cite{DGL06}, for $\mathcal M = \T^2$ or $\mathcal M = S^2$, the authors prove that \eqref{eq:NonlinearSchrodinger_manifoldControlled} for $\lambda >0$ is semiglobally exponentially stabilizable in $H^1(\mathcal M)$ with $h = a(x) (1- \Delta)^{-1} a(x) \partial_t u$ by using propagation of singularities and Strichartz-type estimates from \cite{BGT04}, assuming that $\omega$ contains the union of a neighborhood of the generator circle and a neighborhood of the largest exterior circle of $\T^2$ for $\mathcal M = \T^2$ or $\omega$ contains a neighborhood of the equator for $\mathcal M = S^2$. The authors in \cite{DGL06} also deduce that \eqref{eq:NonlinearSchrodinger_manifoldControlled} is semiglobally exactly controllable in $H^1(\mathcal M)$, by using the semiglobal stabilization and a local exact controllability result. This type of result has been generalized to the situation $\mathcal M = \T^3$, $\mathcal M = S^3$, with the same kind of assumptions, in \cite{Lau10}. It is worth mentioning that \cite{Lau10} uses the Bourgain analysis to handle the semilinearity and also proposes another approach for obtaining the semiglobal exact controllability. Furthermore, \cite{Lau10b} also obtains semiglobal controllablility and stabilizability results for  \eqref{eq:NonlinearSchrodinger_manifoldControlled}, both for focusing and defocusing cases, working at $L^2(\T)$-regularity, with the feedback $h=-i a(x)  u$. More recently, \cite{CCDCAR18} and \cite{YNC21} generalize among other things \cite{DGL06} and \cite{Lau10} with the feedback laws $h = -i a(x) (-\Delta)^{1/2} u$ and $h=-i a(x)u$. For a survey of these results up to 2014, one can read \cite{Lau14}.


\subsection{Main results}
The goal of this paper is to prove the global exponential stabilization of the equation \eqref{eq:NonlinearSchrodinger_manifoldControlled} on the specific case $\mathcal M = \T^d$ for $d \in \{1, 2, 3\}$. Before stating our main results, it is worth mentioning that this work only deals with such dimensions because of the well-posedness result of Proposition~\ref{prop:wellposedbourgain}, see below.

For $d \in \{1, 2, 3\}$, we now consider
\begin{equation}
	\label{eq:NonlinearSchrodinger_torusControlled}
		\left\{
			\begin{array}{ll}
				i  \partial_t u  = -\Delta u + |u|^{2} u + h 1_{\omega} & \text{ in }  (0,+\infty) \times \T^d, \\
				u(0, \cdot) = u_0 & \text{ in } \T^d.
			\end{array}
		\right.
\end{equation}
The stabilization property would be established in the energy space associated to the total energy, the sum of the $L^2$-energy and the nonlinear energy, 
\begin{equation}
\label{eq:totalenergy}
    E(t) = \underbrace{\frac{1}{2}\int_{\T^d} |u(t,x)|^2 dx}_{L^2-\text{energy}} +  \underbrace{\frac{1}{2}\int_{\T^d} |\nabla u(t,x)|^2 dx + \frac{1}{4} \int_{\T^d} |u(t,x)|^{4}  dx}_{\text{nonlinear energy}} \qquad \forall t \geq 0.
\end{equation}
Recall that for $h=0$ in \eqref{eq:NonlinearSchrodinger_torusControlled}, we formally have the conservation law
\begin{equation}
\label{eq:conservationwholeenergy}
     \frac{d}{dt} E(t)= 0 \qquad \forall t \geq 0,
\end{equation}
because the $L^2$-energy is conserved, see \eqref{eq:conservL2}, and the nonlinear energy is conserved, see \eqref{eq:conservEnergyH1}.

Let $\varepsilon \in (0, 2 \pi)$ and assume that $\omega$ is a (nonempty) open subset of $\T^d$ such that by denoting
\begin{equation}
    \label{eq:defIvarepsilon}
    I_{\varepsilon} = (0, \varepsilon) \cup (2 \pi - \varepsilon, 2 \pi) + 2 \pi \Z \subset \T,
\end{equation}
we have
\begin{equation}
    \label{eq:omegaproperties}
    	\left\{
		    \begin{array}{ll}
     \omega_{0} := I_{\varepsilon} \subset \omega & \text{when}\ d= 1,\\
      \omega_{0} := \left(I_{\varepsilon} \times \T\right) \cup \left(\T \times I_{\varepsilon} \right) \subset \omega & \text{when}\ d= 2,\\
      \omega_{0} := \left(I_{\varepsilon} \times \T^2\right) \cup \left(\T \times I_{\varepsilon} \times \T \right) \cup \left( \T^2 \times I_{\varepsilon} \right) \subset \omega & \text{when}\ d= 3,
     \end{array}
		\right.
\end{equation}
When $d=1$, up to a translation, \eqref{eq:omegaproperties} corresponds to the fact that $\omega$ contains an open interval, therefore, we can only assume that $\omega$ is a non-open open subset of $\T^1$. When $d=2$, \eqref{eq:omegaproperties} corresponds to the fact where $\omega$ contains an union of a neighborhood of the generator circle and a neighborhood of the largest exterior circle of $\T^2$. When $d=3$, \eqref{eq:omegaproperties} corresponds to the fact that $\omega$ contains a neighborhood of each face of the cube, the fundamental volume of $\T^3$. It is worth mentioning that such a $\omega$ satisfies in particular GCC.

We consider $a \in C_c^{\infty}(\omega)$ such that $a(x) \geq a_0 > 0$ in $\omega_{0}$. We then look at
\begin{equation}
	\label{eq:NonlinearSchrodingerDamping_torus}
		\left\{
			\begin{array}{ll}
				i  \partial_t u  = -\Delta u + |u|^{2} u - ia(x) u & \text{ in }  (0,+\infty) \times \T^d, \\
				u(0, \cdot) = u_0 & \text{ in } \T^d.
			\end{array}
		\right.
\end{equation}
More precisely, the Section~\ref{sec:wellposedness} is devoted to the well-posedness of \eqref{eq:NonlinearSchrodingerDamping_torus} in Bourgain spaces. 

The main result of the paper is the (uniform) global stabilization of \eqref{eq:NonlinearSchrodingerDamping_torus}.
\begin{theorem}\label{thm:stabilisation_main_result}
Let $d \in \{1, 2, 3\}$. There exist $C,\gamma>0$ such that for all $u_0 \in H^1(\T^d)$, the solution $u$ of \eqref{eq:NonlinearSchrodingerDamping_torus} belongs to $C([0,+\infty);H^1(\T^d))$ and satisfies
\begin{equation}
\label{eq:expdecay}
    E(t) \leq C e^{-\gamma t} E(0)\qquad \forall t \geq 0.
\end{equation}
\end{theorem}
From Theorem \ref{thm:stabilisation_main_result}, one can obtain an estimate of the minimal time of the null-controllability for \eqref{eq:NonlinearSchrodinger_torusControlled}.

From \cite[Theorem 0.2]{Lau10b} in $1$-d, \cite[Theorem 2]{DGL06} in $2$-d, \cite[Theorem 0.1]{Lau10} in $3$-d, we have that for every $u_0 \in H^1(\T^d)$, there exists a time $T>0$ and a control $h \in L^2(0,T;H^1(\T^d))$ such that the solution $u$ of \eqref{eq:NonlinearSchrodinger_torusControlled} belongs to $C([0,T];H^1(\T^d))$ and satisfies $u(T, \cdot) = 0$. So, one can define, for $u_0 \in H^1(\T^d)$, the associated minimal time of controllability, i.e.
\begin{equation}
    T(u_0) = \inf\{T>0\ ;\ \eqref{eq:NonlinearSchrodinger_torusControlled}\ \text{is null-controllable at time }T>0\ \text{from } u_0\}.
\end{equation}
Then, we define the following time of controllability 
\begin{equation}
    \tau(R) = \sup\{T(u_0)>0\ ;\ E(u_0) \leq R\}\qquad \forall R \geq 0.
\end{equation}
The second main result of the paper is an estimate from above of $\tau$.
\begin{theorem}
\label{thm:controlresult}
Let $d \in \{1, 2, 3\}$. There exist $C>0$ sufficiently large such that
\begin{equation}
    \tau(R) \leq C \log(R+1) \qquad \forall R \geq 0.
\end{equation}
\end{theorem}

\textbf{Comments.} Theorems \ref{thm:stabilisation_main_result} and \ref{thm:controlresult} differ from the existing literature. Indeed, Theorem \ref{thm:stabilisation_main_result} states the (uniform) global exponential stabilization of \eqref{eq:NonlinearSchrodinger_torusControlled} by the feedback $h = - i a(x)u$ in the energy space while Theorem \ref{thm:controlresult} states the global null-controllability of \eqref{eq:NonlinearSchrodinger_torusControlled} with an explicit estimate of the (possible) minimal time in function of the size of the initial data in the energy space. Up to our knowledge, they are the first results in this direction for nonlinear Schrödinger equations. The key point is hidden in the exponential decay \eqref{eq:expdecay} where the constants $C>0$, $\gamma>0$ do not depend on the initial data, but only on the geometry of the torus $\T^d$ and the observation set $\omega$. In particular, this estimate answers by the affirmative the open problem stated in \cite[Remark 1]{DGL06}. 

For proving Theorem \ref{thm:stabilisation_main_result}, we develop a new method in comparison to the above mentioned references. Our strategy is based on new quantitative observability inequalities for the cubic defocusing Schrödinger equation with the internal damping \eqref{eq:NonlinearSchrodingerDamping_torus}. The first ingredient for obtaining such inequalities is a new Carleman estimates for 
\eqref{eq:NonlinearSchrodingerDamping_torus}. Instead of seeing $|u|^2 u$ as $V(t,x)u$ with $V$ a time/space-dependent potential then performing a Carleman estimate in a linear Schrödinger-type equation, we directly include the cubic semilinearity in the symmetric part of the Carleman conjugated operator. The internal linear damping $-ia(x) u$ is then treated as a (local) source term that can be absorbed in the Carleman estimates. First, note that such a strategy only enables us to treat defocusing cases, that is strongly in contrast with \cite{Lau10b} where the author can manage to deal with focusing cases. Second, one cannot handle internal linear dampings like $a(x) (1- \Delta)^{-1} a(x) \partial_t u$ or $-i a(x) (-\Delta)^{1/2} u$ that are nonlocal. The second ingredient is a combinaison of energy estimates and Morawetz multipliers method for obtaining the exponential decay of the energy of the solution to the damped equation. As a consequence, we completely bypass the classical use of propagation of compactness-regularity for tackling such a question. We strongly believe that such a method can have other applications to the problem of exponential stabilization of partial differential equations, for instance for semilinear wave equations as considered in \cite{DLZ03}. 

The proof of Theorem \ref{thm:controlresult} is a corollary of the global exponential stabilization from Theorem \ref{thm:stabilisation_main_result} and local controllability results from \cite[Theorem 3.2]{Lau10b} in $1$-d, \cite[Proof of Theorem 2]{DGL06} in $2$-d, \cite[Theorem 0.3]{Lau10} in $3$-d.\\

\textbf{Extensions.} Our main results, i.e. Theorems \ref{thm:stabilisation_main_result} and \ref{thm:controlresult}, can be extended into two directions, that are the geometry of $(\mathcal M, \omega)$ and the semilinearity. For the first point, the key tool is the adaptation of the Carleman estimate in such a setting. For the second point, the key ingredient is the well-posedness of the Cauchy problem \eqref{eq:NonlinearSchrodinger_manifold} in $C([0,+\infty);H^s(\mathcal M))$ for every $s \geq 1$. We mention below the following situations that can be treated with our method:
\begin{itemize}
    \item $d=1$, $\mathcal{M} = \T$, $\omega$ a nonempty open subset of $\T$, replacing $|u|^2 u$ by $|u|^{p-1} u$ for every $p>1$, see \cite[Theorem 2.3]{Bou99} for the well-posedness,
    \item $d=2$, $\mathcal{M} = \T^2$, $\omega$ as in \eqref{eq:omegaproperties}, replacing $|u|^2 u$ by $|u|^{p-1} u$ for every $p>1$, see \cite[Chapter V]{Bou99} for the well-posedness,
    \item $d=2$, $\mathcal{M} = S^2$, $\omega$ a nonempty open subset of $S^2$ containing a neighborhood of $\{x_3 = 0\} \subset \R^3$, replacing $|u|^2 u$ by $|u|^{p-1} u$ for every $p>1$, see \cite[Section 2]{DGL06} for the well-posedness and \cite[Appendix B.1]{Lau10} for the Carleman part,
    \item $d=3$, $\mathcal{M} = \T^3$, $\omega$ as in \eqref{eq:omegaproperties}, replacing $|u|^2 u$ by $|u|^{p-1} u$ for every $p \in (1,5)$, see \cite[Chapter V]{Bou99} for the well-posedness, 
    \item $d=3$, $\mathcal{M} = S^3$, $\omega$ a nonempty open subset of $S^2$ containing a neighborhood of $\{x_4 = 0\} \subset \R^3$, replacing $|u|^2 u$ by $|u|^{p-1} u$ for every $p \in (1,5)$, see \cite[Theorem 1]{BGT05b} for the well-posedness and \cite[Appendix B.1]{Lau10} for the Carleman part.
\end{itemize}
However, we decide for simplicity to focus on the toy model of the cubic defocusing Schrödinger equation on the $d$-dimensional torus.\\

%

\textbf{Open questions. }We finish this part by mentioning some open problems related to Theorems \ref{thm:stabilisation_main_result} and \ref{thm:controlresult}.

From \eqref{eq:expdecay} and the Sobolev embedding $H^1(\T^d) \hookrightarrow L^4(\T^d)$, we can deduce that there exist $C>0$ and $\gamma>0$ such that for every $u_0 \in H^1(\T^d)$, the solution $u$ of \eqref{eq:NonlinearSchrodingerDamping_torus} satisfies 
\begin{equation}
    \label{eq:expdecayh1}
    \norme{u(t,\cdot)}_{H^1(\T^d)}^2 \leq C e^{- \gamma t} \left(\norme{u_0}_{H^1(\T^d)}^2 + \norme{u(t,\cdot)}_{H^1(\T^d)}^4\right),
\end{equation}
But, we do not know if \eqref{eq:expdecayh1} can be replaced with $\norme{u(t,\cdot)}_{H^1(\T^d)}^2 \leq C e^{- \gamma t} \norme{u_0}_{H^1(\T^d)}^2$. Last but not least, from \eqref{eq:expdecayh1}, we obtain in particular the exponential decay of the $L^2$-energy of the solution, but for $H^1(\T^d)$-initial data. In the $1$-d case, where \eqref{eq:NonlinearSchrodingerDamping_torus} is well-posed for initial data in $L^2(\T)$, we may wonder in the spirit of \cite{Lau10b} if $\norme{u(t,\cdot)}_{L^2(\T)}^2 \leq C e^{- \gamma t} \norme{u_0}_{L^2(\T)}^2$, holds true.

In comparison to what is known for the linear case, where an arbitrary open set of the torus $\T^d$ is sufficient for the control, we may wonder to what extent one can weaken the geometrical assumption on $\omega$, that satisfies \eqref{eq:omegaproperties} in our situation. Actually, in the Carleman part, one can consider $\omega$, containing only a neighborhood of the generator circle in $2$-d by taking Carleman weights satisfying only weak pseudoconvexity assumptions as done in \cite{MOR08}. However, this only leads to an observability inequality with a $L^2(\T^d)$-left hand side. This is not sufficient for obtaining the exponential decay of the total energy with our multipliers strategy. 


On the other hand, as said previously, one can extend our main results to the subcritical semilinearities, i.e. $|u|^{p-1} u$ for $p \in (1,5)$ in $3$-d by using the corresponding well-posedness results. Concerning the critical case, it is known from \cite{IP12} that \eqref{eq:NonlinearSchrodinger_manifold} replacing the cubic semilinearity $|u|^{2} u$ by the quintic semilinearity $|u|^{4}u$ is globally well-posed in $H^1(\T^3)$. An interesting open question to adapt our strategy is to prove that the closed-loop equation \eqref{eq:NonlinearSchrodingerDamping_torus} replacing $|u|^{2} u$ by  $|u|^{4}u$ is globally well-posed in $H^s(\T^3)$ for every $s \geq 1$. It seems that it is probably the case but details remain to be written.

Finally, concerning the null-controllability adressed in Theorem \ref{thm:controlresult}, one may ask the question of the uniform large-time controllability, respectively small-time null-controllability, that is there exists a time $T>0$, respectively for every time $T>0$, for every initial data $u_0 \in H^1(\T^d)$, \eqref{eq:NonlinearSchrodinger_torusControlled} is null-controllable at time $T>0$ from $u_0$.

\subsection{Organization of the paper}

The article is organized as follows. In Section \ref{sec:wellposedness}, we prove the well-posedness of \eqref{eq:NonlinearSchrodingerDamping_torus} for initial data in $H^s(\T^d)$ and source terms in $L^2(0,T;H^s(\T^d))$ for $s \geq 1$ with the use of Bourgain spaces. In Section \ref{sec:carlemanpart}, we prove new Carleman estimates then quantitative observability estimates for \eqref{eq:NonlinearSchrodingerDamping_torus} by the use of energy estimates and Morawetz multipliers. In Section \ref{sec:proofmainresults}, we prove the main results of the paper Theorems \ref{thm:stabilisation_main_result} and \ref{thm:controlresult}, in Subsection \ref{sec:expdecay}, we deduce the exponential decay of the total energy of the solution to \eqref{eq:NonlinearSchrodingerDamping_torus} then in Subsection \ref{sec:controlresults}, we obtain the upper bound on the minimal time of the global null-controllability of \eqref{eq:NonlinearSchrodinger_torusControlled}.

\section{Well-posedness results}
\label{sec:wellposedness}

Let $d \in \{1, 2, 3\}$. The goal of this part would be to establish the well-posedness in Bourgain spaces $X_{T}^{s,b}$ for Cauchy problems associated to 
\begin{equation}
	\label{eq:NonlinearSchrodingerDamping_torusg}
		\left\{
			\begin{array}{ll}
				i  \partial_t u  = -\Delta u + |u|^{2} u - ia(x) u + g & \text{ in }  (0,+\infty) \times \T^d, \\
				u(0, \cdot) = u_0 & \text{ in } \T^d.
			\end{array}
		\right.
\end{equation}
where $T>0$, $s \geq 1$, $b \in (1/2,1)$ (depending on $d$ and $s$), $a \in C^{\infty}(\T^d;\R)$, $u_0 \in H^s$, $g \in L^2(0,T;H^s(\T^d))$. Of course, this is an adaptation of the breakthrough idea introduced in \cite{Bou93}, see for instance \cite[Chapter 5]{Bou99} for a detailed account of these techniques. However because we are not able to find the exact result we needed in the literature, we present here the main steps of the proofs for obtaining such a result. We will mainly follow and adapt the presentation given by \cite[Section 1 and Section 2]{Lau10} for the treatment of the $3$-d case to our $d$-dimensional case, $d \in \{1,2,3\}$. Note that in $1$-d, one can directly use \cite[Section 1 and 2]{Lau10b} for this part because the damping terms that we are considering are the same.

In the first part, we introduce the so-called Bourgain spaces $X_{T}^{s,b}$, recall their main properties and present trilinear estimates that will be one of the key point for the well-posedness. In the second part, we present a priori energy identities and estimates for solutions to \eqref{eq:NonlinearSchrodingerDamping_torusg}. In the last part, we state the main result of this part, i.e. Proposition \ref{prop:wellposedbourgain}, that establishes the well-posedness of \eqref{eq:NonlinearSchrodingerDamping_torusg} in $H^s(\T^d)$ for $s \geq 1$. 

\subsection{Some properties of $X^{s,b}$ spaces}

The goal of this part is to state well-posedness results for \eqref{eq:NonlinearSchrodingerDamping_torus} in Sobolev spaces $H^s(\T^d)$. As said previously, it is an adaptation of the idea introduced in \cite{Bou93}, see for instance \cite[Chapter 5]{Bou99} for a detailed account of these techniques. 


In all the following, we use the notation
\begin{equation*}
    \langle x \rangle  = \sqrt{1+x^2}\qquad \forall x \in \R.
\end{equation*}

For $s \in \R$, we equip the Sobolev space $H^s(\T^d)$ with the norm
\begin{equation}
    \label{eq:definitionnormHs}
\norme{u}_{H^s(\T^d)}^2 = \sum_{k \in \Z^d} \langle |k| \rangle^{2s}  |\hat{u}(k)|^2\qquad \forall u \in H^s(\T^d).
\end{equation}
For $s, b \in \R$, the Bourgain space is equipped with the norm
\begin{equation}
    \label{eq:definitionnormeXsb}
    \norme{u}_{X^{s,b}}^2 = \sum_{k \in \Z^d} \int_{\R} \langle |k| \rangle^{2s} \langle \tau+|k|^2 \rangle^{2b} |\hat{\hat{u}}(\tau,k)|^2 d\tau = \norme{u^{\sharp}}_{H^b(\R;H^s(\T^d))}^2,
\end{equation}
where $u^{\sharp}(t) = e^{-it \Delta} u(t)$ and $\hat{\hat{u}}(\tau,k)$ denotes the Fourier transform with respect to the time variable and the spatial variable.

For $T>0$, the restricted Bourgain space $X_T^{s,b}$ is the associated restriction space with the norm
\begin{equation}
    \label{eq:bourgainrestriction}
    \norme{u}_{X_{T}^{s,b}} = \inf\{ \norme{\tilde{u}}_{X^{s,b}}\ ;\ \tilde{u} = u\ \text{in}\ (0,T)\times \T^d\}.
\end{equation}
More generally, for $I$ an interval in $\R$, one can define $X_{I}^{s,b}$ the associated restriction space.\\

The following properties of the Bourgain spaces are straightforward.\\
$\bullet$ The Bourgain spaces $X^{s,b}$ and $X_T^{s,b}$ are Hilbert spaces.\\
$\bullet$ If $s_1 \leq s_2$, $b_1 \leq b_2$, $X^{s_2,b_2}$ is continously embedded in $X^{s_1,b_1}$.\\
$\bullet$ For $b >1/2$, $X_T^{s,b}$ is continously embedded in $C([0,T];H^s(\T^d))$.\\
$\bullet$ For every $s_1< s_2$, $b_1 < b_2$, $X_T^{s_2,b_2}$ is compactly embedded in $X_T^{s_1,b_1}$.\\
$\bullet$ The dual of $X_{T}^{s,b}$ is $X_{T}^{-s,-b}$. \\
$\bullet$ For $\theta \in (0,1)$, the complex interpolation space $(X^{s_1,b_1}, X^{s_2,b_2})_{\theta}$ is $ X^{(1-\theta)s_1 + \theta s_2, (1-\theta)b_1 + \theta b_2}$.\\
$\bullet$ If $s \in \rr$, $b \in (\frac 12, 1)$, $0< T_1 <T_2$, $u_1 \in X^{s,b}_{(0,T_1)}$ and $u_2 \in X^{s, b}_{(T_1, T_2)}$ with $u_1(T_1)=u_2(T_1)$, then the function $u$ defined by $u(t,\cdot)=\left\{\begin{array}{ll}
u_1(t,\cdot), & t\in [0,T_1]\\
u_2(t, \cdot), & t \in [T_1,T_2]
\end{array} \right.$ belongs to $X^{s,b}_{(0,T_2)}$.

The following result studies the stability of the Bourgain spaces with respect to multiplication operators.
\begin{lemma}
\label{lem:mutiplication}
Let $\varphi \in C_c^{\infty}(\R)$, $\psi \in C^{\infty}(\T^d)$, $s \in \R$, $b \in [-1,1]$, $T>0$. The following linear mappings 
\begin{align}
    &\Phi : u \in X^{s,b} \mapsto \varphi(t) u \in X^{s,b},\qquad \Phi_T : u \in X_T^{s,b} \mapsto \varphi(t) u \in X_T^{s,b},\label{eq:multiplicationtemps}\\
    &\Psi : u \in X^{s,b} \mapsto \psi(x) u \in X^{s-|b|,b},\qquad \Psi_T : u \in X_T^{s,b} \mapsto \psi(x) u \in X_T^{s-|b|,b}, \label{eq:multiplicationspatial}
\end{align}
are continuous.
\end{lemma}
\begin{proof}
We only prove the first parts of \eqref{eq:multiplicationtemps} and \eqref{eq:multiplicationspatial}.

By using the commutation of $e^{-it\Delta}$ with $\varphi(t)$, we have
\begin{equation}
    \left\|\varphi u\right\|_{X^{s,b}}= \left\|e^{-it\Delta}[\varphi(t)u]\right\|_{H_t^{b}(H_x^s)}= \left\|\varphi u^{\#}\right\|_{H_t^{b}(H_x^s)}\leq C  \left\|u^{\#}\right\|_{H_t^{b}(H_x^s)}\leq C\left\|u\right\|_{X^{s,b}},
\end{equation}
which concludes the proof of \eqref{eq:multiplicationtemps}.

For \eqref{eq:multiplicationspatial}, we first treat the two cases $b=0$ and $b=1$.

For $b=0$, $X^{s,0}=L^2(\R,H^s)$ and the result is obvious.

For $b=1$, we have $u\in X^{s,1}$ if and only if $u \in L^2(\R,H^s) \textnormal{ and } i\partial_t u+\Delta u \in L^2(\R,H^s)$, with the norm
 $$ \left\| u\right\|^2_{X^{s,1}}= \left\|u\right\|^2_{L^2(\R,H^s)}+\left\|i\partial_t u+\Delta u\right\|^2_{L^2(\R,H^s)}.$$
Then, we have, by using that the commutator $[\psi,\Delta ]$ is an operator of order $1$ in space,
\begin{align*}
   & \left\|\psi(x) u\right\|^2_{X^{s-1,1}}=\left\|\psi u\right\|^2_{L^2(\R,H^{s-1})}+ \left\|i\partial_t (\psi u)+\Delta(\psi u)\right\|^2_{L^2(\R,H^{s-1})}\\
    &\leq C\left(\left\|u\right\|^2_{L^2(\R,H^{s-1})}+ \left\|\psi \left(i\partial_t u+\Delta u\right)\right\|^2_{L^2(\R,H^{s-1})}\right.\left.+ \left\|\left[\psi,\Delta \right]u\right\|^2_{L^2(\R,H^{s-1})}\right)\\
    &\leq C\left(\left\|u\right\|^2_{L^2(\R,H^{s-1})}+ \left\|i\partial_t u+\partial_x^2 u\right\|^2_{L^2(\R,H^{s-1})}+ \left\|u\right\|^2_{L^2(\R,H^{s})}\right)\\
    &\leq C\left\|u\right\|^2_{X^{s,1}},
\end{align*}
that concludes the proof in the case $b=1$.

We finally conclude by interpolation and duality. 
\end{proof}

The following elementary lemma holds, see \cite[Lemma 3.2]{Gin96}.
\begin{lemma}
\label{lem:OneDEstimateSobolev}
Let $\varphi \in C_c^{\infty}(\R)$, $b, b' \in \R$ such that $0 < b' < 1/2 < b,\ b+b' \leq 1$, $T>0$.

If $f \in H^{-b'}(\R)$, then
\begin{equation}
\label{eq:estimationelementaryoneD}
    \norme{t \mapsto \varphi\left(\frac{t}{T} \right)\int_0^{t} f(\tau) d \tau}_{H^b(\R)} \leq C T^{1-b-b'} \norme{f}_{H^{-b'}(\R)}.
\end{equation}
\end{lemma}

One of the key point for establishing well-posedness results associated to the cubic defocusing nonlinear Schrödinger equation consists in establishing the following trilinear estimates.
\begin{proposition}
\label{prop:trilinear}
For every $s_0>1/2$, for every $s_2 \geq s_1 \geq s_0$, there exist $b' \in (0,1/2)$ and $C>0$ such that for every $T \in (0,1)$, $u, v \in X_T^{s_2,b'}$, 
\begin{align}
\label{eq:estimationtrilinear1}
    \norme{|u|^2 u}_{X_T^{s_2,-b'}} &\leq C  \norme{ u}_{X_T^{s_1,b'}}^2 \norme{ u}_{X_T^{s_2,b'}},\\
     \norme{|u|^2 u - |v|^2 v}_{X_T^{s_2,-b'}} &\leq C  \left(\norme{ u}_{X_T^{s_2,b'}}^2  +  \norme{ v}_{X_T^{s_2,b'}}^2 \right)\norme{ u-v}_{X_T^{s_2,b'}}.  \label{eq:estimationtrilinear2}
\end{align} 
\end{proposition}
Proposition \ref{prop:trilinear} is part of the “folklore” for the study of nonlinear Schrödinger equation with periodic boundary conditions. This is a straightforward corollary of \cite[Lemma 0.3]{Lau10b} in $1$-d (one can even take $s_0 = 0$ and $b'=3/8$), \cite[Proposition 2.5 and Proposition 3.5]{BGT05a} in $2$-d (one can even take $s_0>0$) and \cite[Assumption 3, Lemma 1.1]{Lau10} in $3$-d.

\subsection{Energy estimates for smooth solutions}

The following result establish energy and multipliers identities for smooth solutions of \eqref{eq:NonlinearSchrodingerDamping_torus}.
\begin{proposition}
\label{prop:identity_estimates}
Let $T>0$, $s \geq 1$, $a \in C^{\infty}(\T^d;\R)$, $u_0 \in H^s$, $g \in L^2(0,T;H^s(\T^d))$. Assume that $u \in X_{T}^{s,b}$ is a solution of \eqref{eq:NonlinearSchrodingerDamping_torusg} for some $b \in (1/2,1)$, then
\begin{multline}\label{eq:L2_identity}
   \frac{1}{2} \int_{\T^d} |u(t',x)|^2 dx - \frac{1}{2} \int_{\T^d} |u(t, x)|^2 \\= -\int_t^{t'} \int_{\T^d} a(x) |u(s, x)|^2 dx ds + \int_t^{t'} \int_{\T^d} \Im(g(s,x)\overline{u}(s,x)) dx ds \qquad \forall 0 \leq t \leq t' \leq T,
\end{multline}
\begin{multline}\label{eq:energy_identity}
   \frac{1}{2} \int_{\T^d} |\nabla  u(t', x)|^2 dx +\frac{1}{4} \int_{\T^d}|u(t', x)|^{4} dx - \int_{\T^d} |\nabla u(t, x)|^2 dx - \int_{\T^d}\frac{1}{4} |u(t, x)|^{4} dx \\= - \int_t^{t'} \int_{\T^d} a(x)\Im( u(s,x) \partial_t \overline{u}(s,x)) dx ds - \int_t^{t'} \int_{\T^d} \Re(g(s,x)\partial_t \overline{u}(s,x))\qquad \forall 0 \leq t \leq t' \leq T,
\end{multline}
and for $P \in C^{\infty}(\T^d;\R)$,
\begin{multline}\label{eq:multiplier_estimate_Pg}
\int_t^{t'}\int_{\T^d} (\Im(u \partial_t \overline u)-|\nabla u|^2 -|u|^{4}) P(x) dx ds \\
= \frac 12 \int_t^{t'} \int_{\T^d} (\nabla P(x)\cdot \nabla)(|u|^2) dx ds + \int_t^{t'} \int_{\T^d}\Re(g(s,x)P\overline{u}(s,x)) dx ds\qquad \forall 0 \leq t \leq t' \leq T.
\end{multline}
\end{proposition}
Note that the equation \eqref{eq:multiplier_estimate_Pg} is inspired from Morawetz multipliers strategy.
\begin{proof}
In the proof, we assume that the equation \eqref{eq:NonlinearSchrodingerDamping_torusg} is satisfied in the strong sense, that is the case for instance if $u \in X_{T}^{2,b} \subset C([0,T];H^2(\T^d))$. The general case comes from a regularization argument.

For \eqref{eq:L2_identity}, we multiply \eqref{eq:NonlinearSchrodingerDamping_torusg} by $\overline{u}$, integrating in $\T^d$, we get by integration by parts, 
$$   \int_{\T^d} i \partial_t u(s, x) \overline{u(s,x)}  dx - \int_{\T^d} |\nabla u(s, x)|^2  = \int_{\T^d} |u(s, x)|^4  dx- i \int_{\T^d} a(x) |u(s, x)|^2 dx +  \int_{\T^d} g(s,x)\overline{u(s,x)} dx.$$ 
Then, by taking the imaginary part, using $2 \Im (i \partial_t u \overline{u}) = 2\Re(\partial_t u \overline{u}) = \partial_t |u|^2$ to get
$$ \frac{1}{2} \frac{d}{ds}  \int_{\T^d} |u(s,x)|^2 dx  = - \int_{\T^d} a(x) |u(s, x)|^2 dx+  \int_{\T^d} \Im(g(s,x)\overline{u}(s,x)) dx.$$ 
We then integrate for $s \in (t',t)$ to get the result.

For \eqref{eq:energy_identity}, we multiply \eqref{eq:NonlinearSchrodingerDamping_torusg} by $\partial_t \overline{u}$, integrating in $\T^d$, we get by integration by parts
\begin{multline*}
    \int_{\T^d} i |\partial_t u(s, x)|^2  dx - \int_{\T^d} \nabla u(s, x) \cdot \partial_t \nabla \overline{u}(s,x)  \\
    = \int_{\T^d} |u(s, x)|^2 u(s,x) \partial_t \overline{u}(s,x)  dx- i \int_{\T^d} a(x) u(s, x) \partial_t \overline{u}(s,x) dx +  \int_{\T^d} g(s,x)\partial_t \overline{u}(s,x) dx.
\end{multline*}  
Then, by taking the real part, using $2\Re(\nabla u \partial_t \nabla \overline{u}) = \partial_t |\nabla u|^2$, $ 4\Re(|u|^2 u \partial_t \overline{u}) = \partial_t |u|^4$ and $\Re(iz) = -\Im(z)$, we get
\begin{multline*}
     \frac{1}{2} \frac{d}{ds}  \int_{\T^d} |\nabla u(s,x)|^2 dx + \frac{1}{4} \frac{d}{ds}  \int_{\T^d} |u(s,x)|^4 dx\\
    =  -  \int_{\T^d} a(x)\Im( u(s,x) \partial_t \overline{u}(s,x)) dx ds -  \int_{\T^d} \Re(g(s,x)\partial_t \overline{u}(s,x))
\end{multline*} 
We then integrate for $s \in (t',t)$ to get the result.

Let us multiply \eqref{eq:NonlinearSchrodingerDamping_torusg} by $P \overline u$ and integrate on $(t',t) \times \T^d$. By taking the real part and since 
\begin{multline*}
    \Ree{\int_{t'}^t\int_{\T^d} \Delta u P\overline u dxds}=-\int_{t'}^t \int_{\T^d} \Ree{(\nabla u\cdot \nabla P(x) \overline u)} + |\nabla u|^2 P(x) dx ds \\
    = -\int_{t'}^t \int_{\T^d} \frac 12 (\nabla P(x)\cdot \nabla)(|u|^2) + |\nabla u|^2 P(x) dx ds,
\end{multline*}
it follows that $u$ satisfies \eqref{eq:multiplier_estimate_Pg}.
\end{proof}

The next result establishes energy estimates for $E$.
\begin{proposition}
\label{prop:energy_estimates}
Let $T>0$, $s \geq 1$, $a \in C^{\infty}(\T^d;\R)$, $u_0 \in H^s$, $g \in L^2(0,T;H^s(\T^d))$. There there exists $C_{T,\T^d,a}$ such that for $u \in X_{T}^{s,b}$ a solution of \eqref{eq:NonlinearSchrodingerDamping_torusg} for some $b \in (1/2,1)$, then we have
\begin{equation}
    \label{eq:Gronwallg}
    E(t) \leq C \left(E(0)+ \norme{g}_{L^2(0,T;H^1(\T^d))}^2  + \norme{g}_{L^2(0,T;H^1(\T^d))}^4  \right) \qquad \forall t \in [0,T].
\end{equation}
\end{proposition}
\begin{proof}
First, from \eqref{eq:L2_identity}, we deduce from a Gronwall's estimate that
\begin{equation}
\label{eq:gronwallL2}
    \norme{u(t,\cdot)}_{L^2(\T^d)}^2 \leq C \left(\norme{u(0,\cdot)}_{L^2(\T^d)}^2 + \norme{g}_{L^2(0,T;L^2(\T^d))}^2 \right)\qquad\forall t \in [0,T].
\end{equation}
We then use \eqref{eq:multiplier_estimate_Pg} with $P=a$ to estimate the first term in the right hand side of \eqref{eq:energy_identity}, we then have for every $t \in [0,T]$,
\begin{multline}
\label{eq:firstestimatenergy}
    \int_0^{t} \int_{\T^d} a(x)\Ima( u(s,x) \partial_t \overline{u}(s,x)) dx ds \\
    \leq C \int_0^t \int_{\T^d} |\nabla u|^2 dx ds + C \int_0^t \int_{\T^d} |u|^4 dx ds +  C \int_0^t \int_{\T^d} |u|^2 dx ds + C \int_0^t \int_{\T^d} |g|^2 dx ds.
\end{multline}
We then estimate the second term in the right hand side of \eqref{eq:energy_identity} by using the equation \eqref{eq:NonlinearSchrodingerDamping_torusg}, integration by parts, Hölder's estimate and the Sobolev embedding $H^1(\T^d) \hookrightarrow L^4(\T^d)$ and Young's inequality, we then have for every $t \in [0,T]$,
\begin{multline}
\label{eq:secondestimatenergy}
    \int_0^{t} \int_{\T^d} \Re(g\partial_t \overline{u}) = \int_0^{t} \int_{\T^d} \Re(g \overline{i \Delta u - i |u|^2 u + ia(x) u - ig}) \\
    \leq C \Big(\int_0^t \norme{g(s,\cdot)}_{H^1(\T^d)} \norme{u(s,\cdot)}_{H^1(\T^d)}ds +  \int_0^t \norme{g(s,\cdot)}_{L^2(\T^d)} \norme{u(s,\cdot)}_{L^2(\T^d)}ds\\ + \int_0^t \norme{g(s,\cdot)}_{H^1(\T^d)} \norme{u(s,\cdot)}_{L^4(\T^d)}^3 ds + \norme{g}_{L^2(0,T;L^2(\T^d))}^2 \Big)\\
    \leq C \Big(\int_0^t \norme{g(s,\cdot)}_{H^1(\T^d)} (E(s)^{1/2}+E(s)^{3/4})+ \norme{g}_{L^2(0,T;L^2(\T^d))}^2 \Big).
\end{multline}
We plug \eqref{eq:gronwallL2}, \eqref{eq:firstestimatenergy}, \eqref{eq:secondestimatenergy} together with \eqref{eq:energy_identity} to obtain 
\begin{multline*}
    E(t) \leq C \Big(E(0) + \norme{g}_{L^2(0,T;L^2(\T^d))}^2 \\+  \int_{0}^t E(s) ds  +\int_0^t \norme{g(s,\cdot)}_{H^1(\T^d)} (E(s)^{1/2}+E(s)^{3/4}) ds
    \Big) \qquad \forall t \in [0,T].
\end{multline*}
Nonlinear Gronwall's estimate leads to \eqref{eq:Gronwallg}.
\end{proof}

\subsection{Well-posedness results for the nonlinear Schrödinger equation in $H^s(\T^d)$}
Now we can state the well-posedness result for Cauchy problems associated to \eqref{eq:NonlinearSchrodingerDamping_torus}.
\begin{proposition}
\label{prop:wellposedbourgain}
Let $T>0$, $s \geq 1$, $a \in C^{\infty}(\T^d;\R)$. Then there exists $b \in (1/2,1)$ such that for every $u_0 \in H^s$, $g \in L^2(0,T;H^s(\T^d))$, there exists a unique solution $u \in X_{T}^{s,b}$ to \eqref{eq:NonlinearSchrodingerDamping_torusg}.

Moreover, the flow map
\begin{equation}\label{eq:flowmap}
    F : 
    \left|
    \begin{array}{lll}
    H^s(\T^d) \times L^2(0,T;H^s(\T^d)) & \to & X_{T}^{s,b}\\
    (u_0,g) & \mapsto & u,
    \end{array}
    \right.
\end{equation}
is Lipschitz on every bounded subset.
\end{proposition}

\begin{proof}
In Proposition \ref{prop:trilinear}, we fix $s_2 = s$ and $s_1 = 0$. This gives us the existence of $b' \in (0,1/2)$ such that \eqref{eq:estimationtrilinear1} and \eqref{eq:estimationtrilinear2} hold. 

First, we notice that if $g \in L^2(0,T;H^s(\T^d))$ then $g \in X_{T}^{s,-b'}$.

Now, let us fix $b$ such that $b>1/2$ and $b+b' \leq 1$.

\textit{Step 1: Local existence in $X_{T}^{s,b}$.}
We consider the functional
\begin{equation}
    \Phi(u)(t) = e^{it \Delta} u_0 - i \int_0^t e^{i(t-\tau)\Delta} [|u|^2 u - i a(x) u + g](\tau) d \tau.
\end{equation}
Thanks to a Banach fixed-point argument, we will prove that $\Phi$ admits a unique fixed point in $X_{T}^{s,b}$ provided that $T>0$ is small enough.

Let $\psi \in C_c^{\infty}(\R)$ be such that $\psi = 1$ on $[-1,+1]$. Then we have
\begin{equation}
    \norme{\psi(t) e^{i t \Delta} u_0}_{X^{s,b}} = \norme{\psi}_{H^b(\R)} \norme{u_0}_{H^s(\T^d)}.
\end{equation}
Therefore, for $T \leq 1$, we have
\begin{equation}
    \norme{e^{i t \Delta} u_0}_{X_T^{s,b}} \leq C \norme{u_0}_{H^s(\T^d)}.
\end{equation}
The estimate \eqref{eq:estimationelementaryoneD} from Lemma \ref{lem:OneDEstimateSobolev} then implies that
\begin{equation}
    \norme{\psi(t/T) \int_0^t e^{i(t-\tau) \Delta} F(\tau) d \tau}_{X_T^{s,b}} \leq C T^{1-b-b'} \norme{F}_{X_T^{s,-b'}}.
\end{equation}
Then, by using the trilinear estimate \eqref{eq:estimationtrilinear1} from Lemma \ref{prop:trilinear} the multiplication estimate \eqref{eq:multiplicationspatial} from Lemma \ref{lem:mutiplication} and Bourgains spaces embeddings, we get
\begin{align}
    &\norme{\int_0^t e^{i(t-\tau)\Delta} [|u|^2 u - i a(x) u + g](\tau) d \tau}_{X_T^{s,b}}\notag\\
    &\leq C T^{1-b-b'} \norme{|u|^2 u - i a(x) u + g}_{X_T^{s,-b'}}\notag\\
    & \leq C T^{1-b-b'} \left(\norme{|u|^2 u}_{X^{s,-b'}} + \norme{a(x) u}_{X_T^{s,-b'}} + \norme{ g}_{X_T^{s,-b'}}\right)\notag\\
    & \leq C T^{1-b-b'} \left(\norme{|u|^2 u}_{X^{s,-b'}} + \norme{u}_{X_T^{s,b}} + \norme{ g}_{X_T^{s,-b'}}\right)\notag\\
    & \leq C  T^{1-b-b'} \norme{u}_{X_{T}^{s,b}} \left(1 + \norme{u}_{X_{T}^{1,b}}^2\right) + C  T^{1-b-b'} \norme{g}_{X_T^{s,-b'}}.
\end{align}
Then, 
\begin{equation}
\label{eq:balltoball}
    \norme{\Phi(u)}_{X_T^{s,b}} \leq  C \norme{u_0}_{H^s(\T^d)} + C\norme{g}_{X_T^{s,-b'}} + C  T^{1-b-b'} \norme{u}_{X_{T}^{s,b}} \left(1 + \norme{u}_{X_{T}^{1,b}}^2\right).
\end{equation}
Similarly, by using this time the trilinear estimate \eqref{eq:estimationtrilinear2} from Lemma \ref{prop:trilinear} one can prove that 
\begin{equation}
\label{eq:contractiveproperty}
    \norme{\Phi(u)-\Phi(v)}_{X_T^{s,b}} \leq  C  T^{1-b-b'} \norme{u-v}_{X_{T}^{s,b}} \left(1 + \norme{u}_{X_{T}^{s,b}}^2 + \norme{v}_{X_{T}^{s,b}}^2\right).
\end{equation}
By taking $T$ small enough, we then deduce from \eqref{eq:balltoball} and \eqref{eq:contractiveproperty} that $\Phi$ is a contractive mapping from a suitable ball of $X_{T}^{s,b}$ to itself, then admits a unique fixed-point. Moreover, we have uniqueness in the class $X_{T}^{s,b}$ for the Duhamel equation, thanks to \eqref{eq:contractiveproperty} therefore to the Schrödinger equation.\\

\textit{Step 2: Propagation of regularity.}
Let us take here $s>1$. From Step 1, by using that $u_0 \in H^s(\T^d) \hookrightarrow H^1(\T^d)$ and $g \in L^2(0,T;H^s(\T^d)) \hookrightarrow L^2(0,T;H^1(\T^d))$, one can construct a maximal solution $u_1 \in X_{T_1}^{1,b}$ and a maximal solution $u_2 \in X_{T_2}^{s,b}$ for some $T_1, T_2 >0$. We clearly have $T_2 \leq T_1$. Moreover, by uniqueness in $X_{T_1}^{1,b}$, we also have $u_1 = u_2$ in $[0, T_2)$. Assume that $T_2 < T_1$, then there exists $C>0$, $\delta >0$ such that
\begin{equation}
\label{eq:explosionandbounded}
    \lim_{t \to T_2} \norme{u_2}_{X_{t}^{s,b}} = +\infty\ \text{and}\ \norme{u_2}_{X_{t}^{1,b}} \leq C\qquad \forall t \in [T_2-\delta, T_2).
\end{equation}
By using the local existence in $H^1(\T^d)$ i.e. Step 1 (for $s=1$) and gluing of solutions, we then get that there exists $C>0$ such that 
$$\norme{u_2}_{X_{T_2}}^{1,b} \leq C.$$
Then by using \eqref{eq:balltoball} on $[T_2-\varepsilon, T_2]$ for $\varepsilon >0$ small enough such that $C \varepsilon^{1-b-b'} (1+\norme{u}_{X_{T_2}^{1,b}}) < 1/2$, we obtain
\begin{equation}
    \norme{u_2}_{X_{[T_2-\varepsilon, T_2]}^{s,b}} \leq C \norme{u_2(T_2 - \varepsilon, \cdot)}_{H^s(\T^d)} + C \norme{g}_{X_{T_2}^{s,-b'}}.
\end{equation}
Therefore, by using gluing of solutions, we obtain that $u_2 \in X_{T_2}^{s,b}$ contradicting \eqref{eq:explosionandbounded}.\\

\textit{Step 3: Energy estimates for global existence.} The local existence of Step 1 with $s=1$ and the a priori energy estimate \eqref{eq:Gronwallg} implies therefore global existence in $X_{T}^{1,b}$. Step 2 i.e. propagation of regularity then implies the global existence in $X_{T}^{s,b}$ for $s \geq 1$.

The local Lipschitz estimate on the flow is a straightforward consequence of \eqref{eq:contractiveproperty}.
\end{proof}

\section{Carleman estimate on the nonlinear equation}
\label{sec:carlemanpart}

The goal of this part is to obtain a Carleman estimate for the nonlinear Schrödinger equation \eqref{eq:NonlinearSchrodingerDamping_torusg} and to deduce from it an observability inequality. In order to do this, we closely follow the approach of \cite{MOR08} for establishing Carleman estimates for linear Schrödinger equation. We want to highlight the fact that the main difference is the presence of the cubic defocusing nonlinearity $-|u|^2 u$, that we include in our operator. Note that such a strategy has been proposed for instance in the context of dissipative nonlinear parabolic equations in \cite{BGO09} but up to our knowledge, this strategy seems to be new in the context of nonlinear Schrödinger equation.

\subsection{Definition of Carleman weights and main properties}

Let us define $\omega_1 \subset \subset \omega_0 \subset \subset \omega$ such that
by denoting
\begin{equation}
    \label{eq:defIvarepsilon0}
    I_{\varepsilon_0} = (0, \varepsilon_0) \cup (2 \pi - \varepsilon_0, 2 \pi) + 2 \pi \Z \subset \T\qquad \varepsilon_0 \in (0, \varepsilon),
\end{equation}
we have
\begin{equation}
    \label{eq:omegapropertiesBis}
    	\left\{
		    \begin{array}{ll}
     \omega_{0} := I_{\varepsilon_0} \subset \omega & \text{when}\ d= 1,\\
      \omega_{0} := \left(I_{\varepsilon_0} \times \T\right) \cup \left(\T \times I_{\varepsilon_0} \right) \subset \omega & \text{when}\ d= 2,\\
      \omega_{0} := \left(I_{\varepsilon_0} \times \T^2\right) \cup \left(\T \times I_{\varepsilon_0} \times \T \right) \cup \left( \T^2 \times I_{\varepsilon_0} \right) \subset \omega & \text{when}\ d= 3,
     \end{array}
		\right.
\end{equation}

First, we have the following easy lemma.
\begin{lemma}
\label{lem:defeta}
There exists $\eta \in C^{\infty}(\T^d;\R^+)$ such that for some $c>0$
\begin{equation}
\label{eq:propertiesetanabla}
   |\nabla \eta(x)| \geq c > 0\qquad \forall x \in \T^d \setminus \overline{\omega_0},
\end{equation}
\begin{equation}
    \label{eq:strongpseudoconvexity}
    D^2\eta(x)(\xi, \xi) + |\nabla \eta(x) \cdot \xi|^2 \geq c |\xi|^2\qquad \forall (x,\xi) \in (\T^d \setminus \overline{\omega_0})\times \R^d.
\end{equation}
\end{lemma}
\begin{proof}
First, let us define $\chi \in C_c^{\infty}(\T^d)$ such that $\chi = 1$ on $\T^d \setminus \overline{\omega_0}$ and $\chi = 0$ in $\omega_1 \subset \subset \omega_0$.
The function, defined by
\begin{equation}
    \eta(x) = \chi(x) |x|^2 \qquad \forall x \in (0, 2\pi)^d,
\end{equation}
can be extended to a smooth function in $\T^d$ satisfying the two expected properties \eqref{eq:propertiesetanabla}, \eqref{eq:strongpseudoconvexity}.
\end{proof}

Let us define the Carleman weights for $\lambda \geq 1$ a parameter,
\begin{equation}
\label{eq:defalphabeta}
 \alpha(t,x) = \frac{e^{2 \lambda m \|\eta\|_{\infty}} -  e^{\lambda  (\eta(x)+m\norme{\eta}_{\infty})}}{t(T-t)},\qquad \beta(t,x) = \frac{e^{\lambda (\eta(x)+m\norme{\eta}_{\infty})}}{t(T-t)}\qquad \forall (t,x) \in (0,T) \times \T^d,
\end{equation}
where $m>1$ is a fixed number.

\subsection{The Carleman estimate}

The main result of this part is the following Carleman estimate.

\begin{proposition}
\label{prop:Carlemanestimate}
There exist positive constants $C=C(\Omega,\omega)>0$, $\lambda_0 = C > 0$, $s_0 = C( T+ T^2 + T^2 |a|_{\infty} ) >0$, $b \in (1/2,1)$ such that for every $u_0 \in H^1(\T^d)$, $g \in L^2(0,T;H^1(\T^d))$, the solution $u \in X_{T}^{1,b}$ of \eqref{eq:NonlinearSchrodingerDamping_torusg} satisfies
\begin{multline}
\label{eq:Carlemanestimate}
   s^3 \lambda^4 \int_{0}^{T} \int_{\T^d} e^{-2 s \alpha} \beta^3 |u|^2 dx dt + s \lambda \int_{0}^{T} \int_{\T^d} e^{-2 s \alpha} \beta |\nabla u|^2 dx dt + s^2 \lambda^2 \int_{0}^{T} \int_{\T^d} e^{-2 s \alpha} \beta^2 | u|^4 dx dt \\
   \leq C \Big( \int_{0}^{T} \int_{\T^d} e^{-2 s \alpha}  |g|^2 dx dt +  s^3 \lambda^4 \int_{0}^{T} \int_{\omega_0} e^{-2 s \alpha} \beta^3 |u|^2 dx dt\\ + s \lambda \int_{0}^{T} \int_{\omega_0} e^{-2 s \alpha} \beta |\nabla u|^2 dx dt + s^2 \lambda^2 \int_{0}^{T} \int_{\omega_0} e^{-2 s \alpha} \beta^2 | u|^4 dx dt\Big).
\end{multline}
\end{proposition}
\begin{proof}
By using a standard regularization argument using Proposition \ref{prop:wellposedbourgain}, we just need to consider the case where $u \in X_{T}^{2,b}$ so in particular \eqref{eq:NonlinearSchrodingerDamping_torusg} is satisfied in the strong sense. Denote \begin{equation}
\label{eq:psiGamma}
    \psi = e^{-s \alpha} u,\ \Gamma = e^{-s \alpha} g.
\end{equation}
Let us recall that we have
$$ i \partial_t u + \Delta u = |u|^2 u - i a(x) u + g = e^{2 s \alpha} |\psi|^2 e^{s \alpha} \psi - i a(x) e^{s \alpha} \psi + e^{s \alpha} g.$$
We then have
\begin{equation}
\label{eq:P}
    P \psi := i \partial_t \psi + i s \alpha_t \psi + \Delta \psi + 2 s \nabla \alpha \cdot \nabla \psi + s (\Delta \alpha) \psi + s^2 |\nabla \alpha|^2 \psi - e^{2 s \alpha} |\psi|^2 \psi = - i a(x) \psi + \Gamma =: \Gamma_{\psi, g}.
\end{equation}
We decompose $P = P_1 + P_2$ with
\begin{equation}
    P_1 \psi = i s \alpha_t \psi + 2 s \nabla \alpha \cdot \nabla \psi + s (\Delta \alpha) \psi,
\end{equation}
\begin{equation}
    P_2 \psi = i \partial_t \psi +  \Delta \psi +  s^2 |\nabla \alpha|^2 \psi - e^{2 s \alpha} |\psi|^2 \psi.
\end{equation}

For the rest of the proof, we denote $Q_T = (0,T)\times\T^d$ and $q_T = (0,T)\times\omega_0$.

We have
\begin{equation}
    \norme{P_1 \psi + P_2 \psi}_{L^2(Q_T)}^2 = \norme{P_1 \psi}_{L^2(Q_T)}^2 + \norme{ P_2 \psi}_{L^2(Q_T)}^2 + 2 \Re\langle P_1 \psi, P_2 \psi \rangle_{L^2(Q_T)} = \norme{\Gamma_{\psi,g}}_{L^2(Q_T)}^2, 
\end{equation}
therefore
\begin{equation}
\label{eq:boundscalar}
     2 \Re\langle P_1 \psi, P_2 \psi \rangle_{L^2(Q_T)} \leq \norme{\Gamma_{\psi,g}}_{L^2(Q_T)}^2.
\end{equation}
We then decompose
\begin{equation}
\label{eq:expandscalar}
    2 \Re\langle P_1 \psi, P_2 \psi \rangle_{L^2(Q_T)} = I_1 + I_2 + I_3,
\end{equation}
with
\begin{align}
    I_1 & = 2 \Re \left(\int_{Q_T} (2 s \nabla \alpha \cdot \nabla \psi + s (\Delta \alpha) \psi)(- i \partial_t \overline{\psi} + \Delta \overline{\psi} + s^2 |\nabla \alpha|^2 \overline{\psi} - e^{2 s \alpha} |\psi|^2 \overline{\psi}) \right),\\
    I_2 & = 2 \Re \left( \int_{Q_T}i (\partial_t \alpha) \psi(- i \partial_t \overline{\psi} + \Delta \overline{\psi}) \right),\\
    I_3 &= 2 \Re \left(\int_{Q_T} i (\partial_t \alpha) \psi( s^2 |\nabla \alpha|^2 \overline{\psi} - e^{-2 s \alpha} |\psi|^2 \overline{\psi})) \right) = 0.
\end{align}
We first deal with $I_1$, decomposing as follows
\begin{align}
    I_1 & = 2 \Re \left(\int_{Q_T} (2 s \nabla \alpha \cdot \nabla \psi + s (\Delta \alpha) \psi)(\Delta \overline{\psi} + s^2 |\nabla \alpha|^2 \overline{\psi} - e^{-2 s \alpha} |\psi|^2 \overline{\psi}) \right)\\
    &\qquad - 2 \Re \left(\int_{Q_T} i(2 s \nabla \alpha \cdot \nabla \psi + s (\Delta \alpha) \psi) \partial_t \overline{\psi} \right)\\
    &= I_1^1 + I_1^2.
\end{align}

By integration by parts, we have 
\begin{equation}
\label{eq:J}
J:=\int_{Q_T} (\nabla \alpha \cdot \nabla \psi)\Delta \overline{\psi} =-\int_{Q_T} \nabla \overline{\psi}\cdot \nabla (\nabla \alpha \cdot \nabla \psi)).
\end{equation}
Moreover we have
\begin{align}
\label{eq:J1}
\nabla \overline{\psi}\cdot \nabla (\nabla \alpha \cdot \nabla \psi))
&= D^2(\alpha)(\nabla \psi,\nabla\overline{\psi})+ D^2(\psi)(\nabla \overline{\psi},\nabla \alpha),\\
2\Re D^2(\psi)(\nabla \alpha, \nabla \overline{\psi})
&=\nabla \alpha \cdot \nabla |\nabla \psi|^2.\label{eq:J2}
\end{align}
Therefore, from \eqref{eq:J}, \eqref{eq:J1}, \eqref{eq:J2} and an integration by parts, we have
\begin{align}
2\Re J&=-2 \Re\left(\int_{Q_T} D^2(\alpha)(\nabla \psi,\nabla\overline{\psi})\right) - 2 \Re\left(\int_{Q_T} D^2(\psi)(\nabla \alpha, \nabla \overline{\psi})\right)\notag\\
&= -2 \Re\left(\int_{Q_T} D^2(\alpha)(\nabla \psi,\nabla\overline{\psi})\right) +\int_{Q_T} \Delta \alpha \left|\nabla \psi\right|^2.
\end{align}
We can now expand $I_1^1$ as follows, using $\nabla |\psi|^2=2\Re (\overline{\psi} \nabla \psi)$ and $\nabla |\psi|^4=4\Re ( |\psi|^2 \overline{\psi}\nabla \psi)$,
\begin{align}
\notag
I_1^1&=2\Re \Big\{ 2sJ+\int_{Q_T} s(\Delta \alpha)\psi\Delta \overline{\psi} +\int_{Q_T} 2s^3(\nabla \alpha\cdot\nabla \psi) |\nabla \alpha|^2\overline{\psi}\\
&+\int_{Q_T} s^3(\Delta \alpha)|\psi|^2|\nabla \alpha|^2 - \int_{Q_T} 2s(\nabla \alpha\cdot\nabla \psi) e^{2 s \alpha} |\psi|^2 \overline{\psi} - s \int_{Q_T} (\Delta \alpha) e^{2 s \alpha} |\psi|^4 \Big\}\notag\\
&=4s\Re J -2s \Re \int_{Q_T} \left((\nabla \Delta \alpha )\psi+ \Delta \alpha \nabla \psi\right)\cdot \nabla \overline{\psi} -2\int_{Q_T} s^3 \nabla \cdot(|\nabla \alpha|^2 \nabla \alpha) |\psi|^2\notag\\
&  +2\int_{Q_T} s^3(\Delta \alpha)|\psi|^2|\nabla \alpha|^2 + \int_{Q_T} s (\Delta \alpha) e^{2 s \alpha} |\psi|^4 + 2 \int_{Q_T} s^2 |\nabla \alpha|^2  e^{2 s \alpha} |\psi|^4 - 2 \int_{Q_T} s (\Delta \alpha) e^{2 s \alpha} |\psi|^4\notag\\
& =-4s\Re\left(\int_{Q_T} D^2(\alpha)(\nabla \psi,\nabla\overline{\psi})\right) + s  \int_{Q_T} (\Delta^2 \alpha ) |\psi|^2 -2 s^3\int_{Q_T} \nabla \alpha \cdot \nabla (|\nabla \alpha|^2) |\psi|^2\notag\\
& + 2 s^2  \int_{Q_T} |\nabla \alpha|^2  e^{2 s \alpha} |\psi|^4 - s \int_{Q_T}  (\Delta \alpha) e^{2 s \alpha} |\psi|^4.\label{eq:expandI11}
\end{align}
We now compute $I_1^2$, using $2\Re z = z+\overline{z}$, we get by integration by parts
\begin{align*}
-I_1^2&= \int_{Q_T} i(2s\nabla \alpha\cdot\nabla \psi +s(\Delta \alpha)\psi)\overline{\psi}_t -i \int_{Q_T} (2s\nabla \alpha\cdot\nabla \overline{\psi} +s(\Delta \alpha)\overline{\psi} )\psi_t\notag\\
&= \int_{Q_T} -i\left[2s\nabla \alpha_t\cdot\nabla \psi +2s\nabla \alpha\cdot\nabla \psi_t+s(\Delta \alpha_t)\psi+s(\Delta \alpha)\psi_t\right]\overline{\psi} \\
&-i \int_{Q_T} 2s(\nabla \alpha\cdot\nabla \overline{\psi}) \psi_t -i\int_{Q_T} s(\Delta \alpha)\overline{\psi} \psi_t.
\end{align*}
The second term in the right hand side of the previous computation becomes
\begin{align}
-i \int_{Q_T} 2s(\nabla \alpha\cdot\nabla \overline{\psi}) \psi_t= 2is\int_{Q_T} (\Delta \alpha) \overline{\psi} \psi_t + 2is \int_{Q_T} (\nabla\alpha \cdot \nabla \psi_t)\overline{\psi}.
\end{align}
As a consequence, we get
\begin{align}
-I_1^2&= \int_{Q_T} -i2s(\nabla \alpha_t\cdot\nabla \psi )\overline{\psi} -i s\int_{Q_T}(\Delta \alpha_t)|\psi|^2= \int_{Q_T} -i2s(\nabla \alpha_t\cdot\nabla \psi )\overline{\psi} +i s\int_{Q_T} \nabla \alpha_t\cdot\nabla |\psi|^2\notag\\
&=i \int_{Q_T} s\nabla \alpha_t \cdot (\psi\nabla\overline{\psi}-\overline{\psi} \nabla \psi))= 2s\Re \left(i\int_{Q_T} \nabla \alpha_t \cdot (\psi\nabla \overline{\psi}))\right).\label{eq:expandI12}
\end{align}
Finally, we obtain from \eqref{eq:expandI11} and \eqref{eq:expandI12}
\begin{align}
\label{eq:expandI1}
I_1&=-4s\Re\left(\int_{Q_T} D^2(\alpha)(\nabla \psi,\nabla\overline{\psi})\right) + s  \int_{Q_T} (\Delta^2 \alpha ) |\psi|^2 -2 s^3\int_{Q_T} \nabla \alpha \cdot \nabla (|\nabla \alpha|^2) |\psi|^2\notag\\
& + 2 s^2  \int_{Q_T} |\nabla \alpha|^2  e^{2 s \alpha} |\psi|^4 - s \int_{Q_T}  (\Delta \alpha) e^{2 s \alpha} |\psi|^4-2s\Re i\int_{Q_T} \nabla \alpha_t \cdot (\psi\nabla \overline{\psi}))
\end{align}
We now turn to the other term $I_2$, we have
\begin{align}
I_2&=2\Re \int_{Q_T} is \alpha_t \psi(-i\overline{\psi}_t+\Delta \overline{\psi})= s\int_{Q_T} \alpha_t \partial_t|\psi|^2 +2s\Re i \int_{Q_T}  \alpha_t \psi\Delta \overline{\psi}\notag\\
&= - s\int_{Q_T} \alpha_{tt} |\psi|^2 - 2s\Re i \int_{Q_T}  (\nabla\alpha_t \psi + \alpha_t \nabla \psi)\cdot\nabla  \overline{\psi} \notag\\
&= - s\int_{Q_T} \alpha_{tt} |\psi|^2 - 2s\Re  \int_{Q_T}  i(\nabla\alpha_t \cdot \nabla  \overline{\psi}) \psi. \label{eq:expandI2}
\end{align}
Consequently, we get from \eqref{eq:expandscalar}, \eqref{eq:expandI1}, \eqref{eq:expandI2} and using $\nabla \alpha \cdot \nabla |\nabla \alpha|^2=2 D^2(\alpha)(\nabla \alpha,\nabla \alpha)$ that
\begin{align}\notag
2\Re (P_1\psi,P_2\psi)&=\int_{Q_T} \left[-4s^3D^2(\alpha)(\nabla \alpha,\nabla \alpha) - s \alpha_{tt}+s(\Delta^2 \alpha )\right]|\psi|^2 \\
&-4s\Re \int_{Q_T} D^2(\alpha)(\nabla \psi,\nabla\overline{\psi})\notag \\
& +  \int_{Q_T} [s^2 |\nabla \alpha|^2  e^{2 s \alpha}  - s  (\Delta \alpha) e^{2 s \alpha}]|\psi|^4\notag\\
& -4s\Re  \int_{Q_T}  i\psi\nabla\alpha_t \cdot \nabla  \overline{\psi}.
\label{eq:computationsscalar}
\end{align}

The following identities and estimates will be useful in the reminder of the proof
\begin{align}
\nabla \alpha& =- \lambda \beta \nabla\eta, \label{eq:nablaalpha}\\
D^2(\alpha)(X,Y) &=-\beta \lambda \left[D^2(\eta)(X,Y)+\lambda(\nabla \eta \cdot X)(\nabla \eta \cdot Y)\right]\qquad \forall X, Y \in \R^d.\label{eq:Hessalpha}\\
|\nabla \alpha_t| &\leq C T \lambda\beta^2,\quad |\alpha_{tt}| \leq C T^2 \beta^3,\label{eq:alphat}\\
|\Delta \alpha| &\leq C \lambda^2 \beta,\quad |\Delta^2 \alpha| \leq C \lambda^4 \beta.\label{eq:deltaalpha}
\end{align}
Then, we have from the properties of the weight \eqref{eq:propertiesetanabla}, \eqref{eq:strongpseudoconvexity} and \eqref{eq:nablaalpha}, \eqref{eq:Hessalpha}, 
\begin{align}
-4s^3D^2(\alpha)(\nabla \alpha,\nabla \alpha)&=4s^3\lambda \beta  \left[D^2(\eta)(\nabla\alpha,\nabla\alpha)+\lambda\left|\nabla \eta \cdot \nabla \alpha\right|^2\right] \geq c s^3 \lambda^4 \beta^3\ \text{in}\ \T^d \setminus \omega_0,\label{eq:termpsicarre}\\
- 4 s D^2(\alpha)(X,X)&=  s\lambda \beta \left[D^2(\eta)(X,X)+\lambda|\nabla \eta \cdot X|^2\right] \geq cs \lambda \beta |X|^2 \ \text{in}\ \T^d \setminus \omega_0,\ \forall X \in \R^d,\label{eq:termgradientpsicarre}\\
s^2 |\nabla \alpha|^2   & \geq cs^2 \lambda^2 \beta^2 \ \text{in}\ \T^d \setminus \omega_0.\label{eq:termpsi4}
\end{align}
We then have from \eqref{eq:computationsscalar}, \eqref{eq:P}, \eqref{eq:boundscalar} and \eqref{eq:termpsicarre}, \eqref{eq:termgradientpsicarre}, \eqref{eq:termpsi4},
\begin{align}
&s^3 \lambda^4 \int_{Q_T} \beta^3 |\psi|^2 + s \lambda \int_{Q_T} \beta |\nabla \psi|^2 + s^2 \lambda^2 \int_{Q_T} \beta^2 e^{2 s \alpha} |\psi|^4 \notag\\
& \leq C\Bigg(\int_{Q_T} |\Gamma|^2 + \int_{Q_T} |a|^2 |\psi|^2+\left| \int_{Q_T} \left[- s \alpha_{tt}+s(\Delta^2 \alpha )\right]|\psi|^2\right| + \left| \int_{Q_T} s  (\Delta \alpha) e^{2 s \alpha}|\psi|^4\right| \notag \\
&+ \left|4s\Re  \int_{Q_T}  i\psi\nabla\alpha_t \cdot \nabla  \overline{\psi}\right| \notag\\
& + s^3 \lambda^4 \int_{q_T} \beta^3 |\psi|^2 + s \lambda \int_{q_T} \beta |\nabla \psi|^2 + s^2 \lambda^2 \int_{q_T} \beta^2 e^{2 s \alpha} |\psi|^4\Bigg).\label{eq:finalcomputationsscalarbeforeabsorb}
\end{align}
Now, we will absorb some right hand side terms in \eqref{eq:finalcomputationsscalarbeforeabsorb}. Take $\varepsilon>0$ a small positive number and $C_{\varepsilon}>0$ a positive constant depending only of $\varepsilon$ that can vary from one line to another, we have from \eqref{eq:alphat}, \eqref{eq:deltaalpha} that
\begin{equation}
\label{eq:absorba}
\int_{Q_T} |a|^2 |\psi|^2 \leq |a|_{\infty}^2 \int_{Q_T} |\psi|^2 \leq \varepsilon s^3 \lambda^4 \int_{Q_T} \beta^3 |\psi|^2\ \text{for}\ s \geq C_{\varepsilon} T^2 |a|_{\infty}^{2/3},
\end{equation}
\begin{equation}
\label{eq:absborbalphatt}
  s  \int_{Q_T} |\alpha_{tt}| |\psi|^2 \leq C s T^2 \int_{Q_T} \beta^3 |\psi|^2 \leq \varepsilon s^3 \lambda^4 \int_{Q_T} \beta^3 |\psi|^2 \ \text{for}\ s^2 \geq C_{\varepsilon} T^2\ \text{i.e.}\ s \geq C_{\varepsilon} T,
\end{equation}
\begin{equation}
\label{eq:absborbdeltaalpha}
  s  \int_{Q_T} |\Delta^2 \alpha| |\psi|^2 \leq C s \lambda^4 \int_{Q_T} \beta |\psi|^2 \leq \varepsilon s^3 \lambda^4 \int_{Q_T} \beta^3 |\psi|^2 \ \text{for}\ s^2 \geq C_{\varepsilon} \beta^{-2}\ \text{i.e.}\ s \geq C_{\varepsilon} T^2,
\end{equation}
\begin{equation}
\label{eq:absborbdeltaalphapsi4}
  s  \int_{Q_T} |\Delta \alpha| e^{2 s \alpha} |\psi|^4 \leq C s \lambda^2 \int_{Q_T} \beta e^{2 s \alpha} |\psi|^4 \leq \varepsilon s^2 \lambda^2 \int_{Q_T} \beta¨2 e^{2 s \alpha} |\psi|^4 \ \text{for}\ s \geq C_{\varepsilon} \beta^{-1}\ \text{i.e.}\ s \geq C_{\varepsilon} T^2,
\end{equation}
\begin{align}
\left|4s\Re  \int_{Q_T}  i\psi\nabla\alpha_t \cdot \nabla  \overline{\psi}\right| &\leq  C  T s \lambda \int_{Q_T} \beta^2 |\nabla \psi| |\psi|\notag\\
& \leq \varepsilon s \lambda \int_{Q_T} \beta |\nabla \psi|^2 + C_{\varepsilon} s \lambda T^2 \int_{Q_T} \beta^3 |\psi|^2\notag\\
&\leq  \varepsilon s \lambda \int_{Q_T} \beta |\nabla \psi|^2 +  \varepsilon s^3 \lambda^4 \int_{Q_T} \beta^3 |\psi|^2\ \text{for}\ s \geq C_{\varepsilon} T.
\label{eq:absborblast}
\end{align}
We finally get from \eqref{eq:absorba}, \eqref{eq:absborbalphatt}, \eqref{eq:absborbdeltaalpha}, \eqref{eq:absborbdeltaalphapsi4}, \eqref{eq:absborblast} and \eqref{eq:finalcomputationsscalarbeforeabsorb} that 
\begin{align}
&s^3 \lambda^4 \int_{Q_T} \beta^3 |\psi|^2 + s \lambda \int_{Q_T} \beta |\nabla \psi|^2 + s^2 \lambda^2 \int_{Q_T} \beta^2 e^{2 s \alpha} |\psi|^4\notag\\
& \leq C\Bigg(\int_{Q_T} |\Gamma|^2 +   s^3 \lambda^4 \int_{q_T} \beta^3 |\psi|^2 + s \lambda \int_{q_T} \beta |\nabla \psi|^2 + s^2 \lambda^2 \int_{q_T} \beta^2 e^{2 s \alpha} |\psi|^4\Bigg). \label{eq:afterabsorb}
\end{align}

Now we reuse the expression of $\psi$ in function of $u$ and $\Gamma$ in function of $g$ given in \eqref{eq:psiGamma} to get the desired Carleman estimate \eqref{eq:Carlemanestimate}.

\end{proof}

\subsection{From the Carleman estimate to the observability inequality}
The goal of this part is to obtain an observability inequality for \eqref{eq:NonlinearSchrodingerDamping_torusg}, starting from the Carleman estimate previously obtained in Proposition \ref{prop:Carlemanestimate} and energy, multipliers estimates stated in Proposition \ref{prop:identity_estimates}.

\begin{proposition}\label{prop:observabilityinequality}
There exist a positive constant $C=C(\Omega,\omega,|a|_{\infty})>0$ and $b \in (1/2,1)$ such that for every $u_0 \in H^1(\T^d)$, the solution $u \in X_{T}^{1,b}$ of \eqref{eq:NonlinearSchrodingerDamping_torusg} satisfies
\begin{equation}
    \label{eq:ObsSchrodinger}
    \forall t \in [0,T], \quad E(t) \leq \exp\left(C\left(1+\frac{1}{T}\right)\right) \int_{0}^{T} \int_{\T^d}(|u|^2 + |\nabla u|^2 +|u|^{4}  ) a(x) dxds.
\end{equation}
\end{proposition}

\begin{proof}
From the properties of the weights \eqref{eq:defalphabeta} and from the choices of $\lambda$, $s$ in Proposition \ref{prop:Carlemanestimate}, we deduce that
\begin{align*}
    e^{-2 s \alpha} (\beta +\beta^2+\beta^3)  &\geq \exp\left(- C\left(1+\frac{1}{T} \right)\right)\ \text{in}\ (T/4,3T/4) \times \T^d,\\
     e^{-2 s \alpha} (1+\beta^3 + \beta + \beta^2) &\leq C\left(1+\frac{1}{T^6}\right)\ \text{in}\ (0,T) \times \T^d.
\end{align*}
We then obtain from the Carleman estimate \eqref{eq:Carlemanestimate} and the property of $a$ in $\omega_0$ that
\begin{equation}
\label{eq:obsT4}
\int_{T/4}^{3T/4} E(t) dt 
\leq \exp\left(C\left(1+\frac{1}{T} \right)\right) \int_0^T \int_{\T^d} (|u|^2 + |\nabla u|^2 +|u|^{4}  ) a(x) dx dt.
\end{equation}
As a first step, let us show that 
\begin{equation}\label{eq:obsL2}
\forall t \in [0,T], \quad \int_{\T^d} |u(t,x)|^2 dx \leq \exp\left(C\left(1+\frac{1}{T}\right)\right) \int_{0}^{T} \int_{\T^d}(|u|^2 + |\nabla u|^2 +|u|^{4}  ) a(x) dxds.
\end{equation}
Indeed, thanks to \eqref{eq:L2_identity}, we have for all $t, t' \in [0,T]$, 
\begin{equation*}
    \int_{\T^d} |u(t, x)|^2 dx \leq \int_0^T \int_{\T^d} a(x) |u(s,x)|^2 dx ds + \int_{\T^d} |u(t',x)|^2 dx.
\end{equation*}
By integrating on $\{T/4 \leq t' \leq 3T/4\}$, we deduce from the above estimate together with \eqref{eq:obsT4} that
\begin{equation*}
    \int_{\T^d} |u(t, x)|^2 dx \leq \left(1+\frac 2T\exp\left(C\left(1+ \frac{1}{T}\right)\right)\right) \int_{0}^{T} \int_{\T^d}(|u|^2 + |\nabla u|^2 +|u|^{4}  ) a(x) dxds,
\end{equation*}
which proves \eqref{eq:obsL2} for a suitable constant $C>0$.

By now, let us deal with the whole energy $E(t)$. Notice that from the identities \eqref{eq:L2_identity} and \eqref{eq:energy_identity}, we have for all $0 \leq t \leq t' \leq T$,
\begin{equation*}
    E(t')-E(t) = -\int_t^{t'} \int_{\T^d} a(x) |u(x,s)|^2 dx ds -\int_t^{t'} \int_{\T^d} a(x) \Ima{(u \partial_t \overline u)} ds dx.
\end{equation*}
Moreover, by using \eqref{eq:multiplier_estimate_Pg} with $P=a$, this leads to 
\begin{align*}
    &E(t')-E(t)\\
    & = -\int_t^{t'} \int_{\T^d} a(x) |u(x,s)|^2 dx ds -\frac 12 \int_t^{t'} \int_{\T^d} \nabla a(x)\cdot \nabla(|u|^2) dx ds -\int_t^{t'} \int_{\T^d} a(x) (|\nabla u|^2 +|u|^4) dx ds\\
    &= -\int_t^{t'} \int_{\T^d} a(x) |u(x,s)|^2 dx ds +\frac 12 \int_t^{t'} \int_{\T^d} \Delta a(x)  |u|^2 dx ds -\int_t^{t'} \int_{\T^d} a(x) (|\nabla u|^2 +|u|^4) dx ds.
\end{align*}
We then deduce that for all $t, t' \in [0,T]$, 
\begin{equation}
    E(t) \leq \int_0^{T} \int_{\T^d} a(x) (|u|^2+|\nabla u|^2 +|u|^4) dx ds + \frac{\|\Delta a\|_{L^{\infty}}}2 \int_0^T \int_{\T^d}|u(s,x)|^2 dxds + E(t').
\end{equation}
After integrating on $\{T/4 \leq t' \leq 3T/4\}$, the conclusion of Proposition \ref{prop:observabilityinequality} follows from \eqref{eq:obsT4} and \eqref{eq:obsL2}.
\end{proof}
From Proposition \ref{prop:observabilityinequality}, we finally obtain the following useful result.
\begin{corollary}
There exist a positive constant $C=C(\Omega,\omega,|a|_{\infty})>0$ and $b \in (1/2,1)$ such that for every $u_0 \in H^1(\T^d)$, the solution $u \in X_{T}^{1,b}$ of \eqref{eq:NonlinearSchrodingerDamping_torusg} satisfies
\begin{equation}\label{eq:energy_observability}
E(0) +\int_0^T E(t) dt \leq \exp\left(C\left(1+\frac{1}{T}\right)\right) \int_0^T \int_{\T^d} a(x) (|u(t,x)|^2 +|\nabla u(t,x)|^2 +|u(t,x)|^{4}) dx dt.
\end{equation}
\end{corollary}

\section{Proof of the main results}
\label{sec:proofmainresults}

\subsection{Exponential decay of the solution to the nonlinear equation}
\label{sec:expdecay}

The goal of this part is to prove Theorem \ref{thm:stabilisation_main_result}. 

We first state a technical lemma that would be useful in the sequel, it comes from \cite[Lemma~4.4]{YNC21}.
\begin{lemma}\label{lem:majoration_derivee}
Let $a \in \mathcal C^1(\T^d)$ be a non-negative real function. For all $\varepsilon>0$, there exists a positive constant $C_{\varepsilon}>0$ such that 
$$\forall x \in \T^d, \quad |\nabla a(x)|^2 \leq C_{\varepsilon} a(x)+ \varepsilon.$$
\end{lemma}
\begin{proof}
Let us proceed by contradiction and assume that there exist $\varepsilon_0>0$ and a sequence $(x_n)_{n \in \nn} \subset \T^d$ such that for all $n \in \nn$,
\begin{equation}\label{eq:contradiction_nabla_a}
|\nabla a(x_n)|^2 \geq n a(x_n)+ \varepsilon_0.
\end{equation}
Up to a subsequence, we can assume that $(x_n)_{n \in \nn}$ tends to some $x_\infty \in \T^d$. Since $a$ is non-negative and $\varepsilon_0>0$, \eqref{eq:contradiction_nabla_a} implies that $a(x_\infty)=0$. In particular, $x_\infty$ minimizes $a$ and we obtain $\nabla a(x_\infty)=0$. The contradiction then follows from the fact that $0< \varepsilon_0 \leq |\nabla a(x_\infty)|^2=0.$
\end{proof}

\begin{proof}[Proof of Theorem~\ref{thm:stabilisation_main_result}]
We first express the right hand side of the observability estimate \eqref{eq:energy_observability} thanks to the total energy of the system. We proceed as follows. From \eqref{eq:L2_identity}, we have
\begin{equation*}
 \int_0^{T} \int_{\T^d} a(x) |u(s, x)|^2 dx ds =  \frac{1}{2} \int_{\T^d} |u(0,x)|^2 dx - \frac{1}{2} \int_{\T^d} |u(T, x)|^2 dx,
\end{equation*}
From \eqref{eq:multiplier_estimate_Pg} with $P=a$ together with \eqref{eq:energy_identity}, we have
\begin{multline*}
\int_0^{T}\int_{\T^d} a(x)( |\nabla u|^2+  |u|^{4}) dx ds 
= \int_0^{T}\int_{\T^d} a(x) (\Im(u \partial_t \overline u)) dx ds - \frac 12 \int_0^{T} \int_{\T^d} (\nabla a(x)\cdot \nabla)(|u|^2) dx ds \\
= \frac{1}{2} \int_{\T^d} |\nabla u(0,x)|^2 dx + \frac{1}{4} \int_{\T^d}|u(0, x)|^{4} dx - \frac{1}{2}\int_{\T^d} |\nabla u(T, x)|^2 dx - \frac{1}{4}  \int_{\T^d}|u(T, x)|^{4} dx \\
 - \frac 12 \int_0^{T} \int_{\T^d} (\nabla a(x)\cdot \nabla)(|u|^2) dx ds
\end{multline*}
We sum the last two previous identities and we use the observability estimate \eqref{eq:energy_observability} to get that
%
\begin{equation}\label{eq:energy_obs_estimate1}
    E(0) + \int_0^T E(t)dt \leq \frac{C_T}2\int_0^T \int_{\T^d} |\nabla a(x)| |\nabla u(t,x)||u(t,x)| dxdt + C_T(E(0)-E(T)).
\end{equation}
Let $\varepsilon>0$ to be chosen later. According to Lemma~\ref{lem:majoration_derivee}, there exists a positive constant $C_{\varepsilon}>0$ such that 
$$\forall x \in \T^d, \quad |\nabla a(x)|^2 \leq C_{\varepsilon} a(x) +\varepsilon.$$
We therefore deduce from \eqref{eq:energy_obs_estimate1} and the $L^2$-identity \eqref{eq:L2_identity} that there exists a new constant $C'_{\varepsilon}>0$ such that
\begin{align}
    &E(0) +\int_0^T E(t) dt\notag\\
    &\leq \frac{C_T}{2} \left( C'_{\varepsilon}\int_0^T \int_{\T^d}  a(x) | u(t,x)|^2 dx dt + E(0)-E(T)+\varepsilon \int_0^T E(t)dt \right) \nonumber\\
    & \leq \frac{C_T}{2} \left( \frac{C'_{\varepsilon}}2 \left(\|u(0,\cdot)\|^2_{L^2(\T^d)}-\|u(T,\cdot)\|^2_{L^2(\T^d)}\right) + E(0)-E(T)+\varepsilon \int_0^T E(t)dt \right).
\end{align}
By now, we choose $\varepsilon =C_T^{-1}$. This readily provides 
\begin{equation}\label{eq:energy_obs_estimate2}
    E(0) +\int_0^T E(t) dt \leq C_T \left( \tilde{C}_T\left(\|u(0,\cdot)\|^2_{L^2(\T^d)}-\|u(T,\cdot)\|^2_{L^2(\T^d)}\right) + E(0)-E(T) \right),
\end{equation}
where $\tilde{C}_T>0$ is a new positive constant depending only on $T$.

Let us define an auxiliary energy by 
$$\forall t \geq 0, \quad \tilde{E}(t) =E(t)+\tilde{C}_T \|u(t,\cdot)\|^2_{L^2(\T^d)},$$
which satisfies for all $t \geq 0$,
\begin{equation*}
    E(t) \leq \tilde{E}(t) \leq (1+\tilde{C}_T) E(t).
\end{equation*}
From \eqref{eq:energy_obs_estimate2}, we have
\begin{equation}
    \tilde{E}(0) \leq \hat{C}_T(\tilde{E}(0)-\tilde{E}(T)),
\end{equation}
where $\hat{C}_T= (1+\tilde{C}_T) C_T$. This last inequality directly implies that 
$$\tilde{E}(T) \leq \frac{\hat{C}_T-1}{\hat{C}_T} \tilde{E}(0).$$
On the other hand, thanks to \eqref{eq:L2_identity}, \eqref{eq:energy_identity}, \eqref{eq:multiplier_estimate_Pg} and the Gronwall's inequality from Proposition \ref{prop:energy_estimates}, one can readily obtain that there exists a positive constant $M_T>0$ such that
$$\forall 0\leq t \leq T, \quad \tilde{E}(t) \leq M_T \tilde{E}(0).$$
Finally, we obtain that there exists two positive constants $K, \gamma>0$ such that for all $t\geq 0$,
\begin{equation*}
    \tilde E(t) \leq K e^{-\gamma t} \tilde E(0)
\end{equation*}
and then,
\begin{equation*}
    \forall t \geq 0, \quad E(t) \leq (1+\tilde{C}_T)Ke^{-\gamma t} E(0).
\end{equation*}
This concludes the proof of Theorem~\ref{thm:stabilisation_main_result}.
\end{proof}

\subsection{Global null-controllability of the nonlinear equation}
\label{sec:controlresults}

This section is devoted to the proof of the Theorem~\ref{thm:controlresult}. We adopt the classical strategy (see for example \cite{Lau10b}, \cite{Lau14}) which consists in using our stabilisation result Theorem~\ref{thm:stabilisation_main_result} after proving a local controllability result near to $0$.

Let $\varphi \in \mathcal C^{\infty}_c(0,T)$ be a nonnegative function different from zero.

\subsubsection{First step: study of the linear system.}
Before studying the local controllability of nonlinear equation, let us consider the linear system
\begin{equation}
	\label{eq:freelinearSchrodinger_torus}
		\left\{
			\begin{array}{ll}
				i  \partial_t \Psi  = -\Delta \Psi + a^{2}(x) \varphi^2(t) e^{it\Delta} \phi_0 & \text{ in }  (0,+\infty) \times \T^d, \\
				\Psi(T, \cdot) = 0 & \text{ in } \T^d,
			\end{array}
		\right.
\end{equation}
for $\phi_0 \in L^2(\T^d)$.
Let us define the linear operator 
$$\begin{array}{llll}
S :& L^2(\T^d) &\longrightarrow & L^2(\T^d)\\
    & \phi_0 & \longmapsto & \Psi(0,\cdot).
\end{array}$$
where $\Psi$ is the mild solution of \eqref{eq:freelinearSchrodinger_torus}. One can easily check that $S$ is an injective continuous map. Let us highlight that the surjectivity of $S$ would lead to the exact controllability of the linear system \eqref{eq:freelinearSchrodinger_torus}. Thanks to the Hilbert Uniqueness Method, the question of its surjectivity is equivalent to the observability estimates
$$\exists C_{a,\varphi} >0, \forall u_0 \in L^2(\T^d), \quad \|u_0\|^2_{L^2(\T^d)} \leq C_{a, \varphi} \int_{\rr} \varphi(t)^2 \|a e^{it\Delta} u_0\|^2_{L^2(\T^d)} dt,$$
which are known to hold in any dimension, see \cite[Theorem~4]{AM14}.
The linear map $S$ is therefore an isomorphism from $L^2(\T^d)$ to $L^2(\T^d)$. Actually, the following proposition states that $S$ is also an isomorphism from $H^1(\T^d)$ to $H^1(\T^d)$.
\begin{proposition}[{\cite[Lemma~3.1]{Lau10b}}]\label{prop:control_operator}
The Sobolev space $H^1(\T^d)$ is $S$ invariant and $S : H^1(\T^d) \longrightarrow H^1(\T^d)$ is an isomorphism.
\end{proposition}
Let us mention that Proposition~\ref{prop:control_operator} is proved in \cite[Lemma~3.1]{Lau10b} in the one-dimensional setting. However the strategy adopted by the author \cite{Lau10b} can be easily adapted in any dimension $d \geq 1$.
\subsubsection{Second step: controlability near to $0$.}
As a first step, we prove the following proposition:

\begin{proposition}\label{prop:local_control}
Let $b \in (1/2, 1)$ and $T>0$. There exists $\varepsilon>0$ such that for all $u_0 \in H^1(\T^d)$ satisfying $\|u_0\|_{H^1(\T^d)} \leq \varepsilon$ there exists $g \in \mathcal C([0,T], H^1(\T^d))$ supported in $[0,T] \times \overline{\omega}$ so that the unique solution $u \in X^{1,b}_T$ of
\begin{equation}
	\label{eq:NonlinearSchrodinger_torusControllednear0}
		\left\{
			\begin{array}{ll}
				i  \partial_t u  = -\Delta u + |u|^{2} u + g 1_{\omega} & \text{ in }  (0,+\infty) \times \T^d, \\
				u(0, \cdot) = u_0 & \text{ in } \T^d,
			\end{array}
		\right.
\end{equation}
satisfies $u(T,\cdot)=0$.
\end{proposition}
\begin{proof}
For $\phi_0 \in H^1(\T^d)$, we consider $u \in X^{1, b}_T$ the unique solution of 
\begin{equation}
	\label{eq:NonlinearSchrodinger_toruscontrolled}
		\left\{
			\begin{array}{ll}
				i  \partial_t u  = -\Delta u + |u|^{2} u + a^2(x) \varphi^2(t) e^{it \Delta}\phi_0 & \text{ in }  (0,+\infty) \times \T^d, \\
				u(T, \cdot) = 0 & \text{ in } \T^d,
			\end{array}
		\right.
\end{equation}
$v \in X^{1,b}_T$ the unique solution of
\begin{equation}
	\label{eq:NonlinearSchrodinger_torus_nonlinearpart}
		\left\{
			\begin{array}{ll}
				i  \partial_t v  = -\Delta v + |u|^{2} u  & \text{ in }  (0,+\infty) \times \T^d, \\
				v(T, \cdot) = 0 & \text{ in } \T^d,
			\end{array}
		\right.
\end{equation}
and define $L\phi_0=u(0)$ and $K\phi_0=v(0)$. We therefore have 
$$\forall \phi_0 \in H^1(\T^d), \quad L\phi_0 = K\phi_0 +S\phi_0.$$

Our goal is to show that there exists $\eta>0$ such that $B_{H^1}(0,\eta) \subset \Ima (L)$.
Notice that the equation $u_0 =L\phi_0$ is equivalent to 
$$\phi_0= S^{-1} u_0 -S^{-1} K \phi_0$$
and this question is then equivalent to find a fixed point of $$B\phi_0:=S^{-1} u_0 -S^{-1} K \phi_0,$$ for $u_0$ sufficiently small in $H^1(\T^d)$. 

Let $0<\eta, \varepsilon \leq 1$ be two small parameters to be chosen later and $u_0 \in B_{H^1}(0, \eta)$. Without loss of generality, we can assume $T\leq 1$.
Since $S : H^1(\T^d) \longrightarrow H^1(\T^d)$ is an isomorphism, we have that for all $\phi_0 \in H^1(\T^d)$,
\begin{align*}\|B\phi_0 \|_{H^1(\T^d)} & \leq C ( \|u_0\|_{H^1(\T^d)}+\|Ku_0\|_{H^1(\T^d)})\\
& = C ( \|u_0\|_{H^1(\T^d)}+\|v(0, \cdot)\|_{H^1(\T^d)}).
\end{align*}
Moreover, we have thanks to the continuous embedding of $X^{1,b}_T$ in $\mathcal C([0,T], H^1(\T^d))$ and Lemma~\ref{lem:OneDEstimateSobolev},
\begin{align*}
    \|v(0, \cdot)\|_{H^1(\T^d)} & \leq C \|v\|_{X^{1,b}_T} \\
    & \leq C T^{1-b-b'} \||u|^2u\|_{X^{1,-b'}}\\
    & \leq C \|u\|^3_{X^{1,b'}_T} \leq \|u\|^3_{X^{1,b}_T}.
\end{align*}
Furthermore, by using the fact that the flow map, defined by \eqref{eq:flowmap}, is Lipschitz on the bounded set $B_{H^1(\T^d)}(0,1) \times B_{L^2(0,T; H^1(\T^d))}(0,1)$, we obtain for all $\phi_0 \in \overline{B_{H^1(\T^d)}}(0, \varepsilon)$, 
$$\|u\|_{X^{1,b}_T} \leq C \|\phi_0\|_{H^1(\T^d)} \leq C \varepsilon.$$
As a consequence, we deduce that for all $\phi_0 \in \overline{B_{H^1(\T^d)}}(0, \varepsilon)$,
$$\|B\phi_0\|_{H^1(\T^d)} \leq C (\eta + \varepsilon^3),$$
for some positive constant $C>0$ independent on $\varepsilon$ and $\eta$. We can therefore choose $\varepsilon_0>0$ such that 
$$\varepsilon^3_0 \leq \frac{\varepsilon_0}{2C},$$
and $\eta= \frac{\varepsilon_0}{2C}$,
and we obtain that the closed ball $\overline{B_{H^1(\T^d)}}(0, \varepsilon_0)$ is $B$ invariant. It remains to check that $B$ is a contraction mapping on this ball.
Let $\phi_0, \phi_1 \in \overline{B_{H^1(\T^d)}}(0, \varepsilon_0)$. We have
\begin{align*}
\|B\phi_0-B\phi_1\|_{H^1(\T^d)} & = \|S^{-1}(K\phi_0-K\phi_1)\|_{H^1(\T^d)}  \\
& \leq C \|v_0(0, \cdot)-v_1(0, \cdot)\|_{H^1(\T^d)}\\
& \leq C\|v_0-v_1\|_{X^{1, b}_T} \\
& \leq C T^{1-b-b'}\||u_0|^2 u_0-|u_1|^2u_1\|_{X^{1,-b'}_T},
\end{align*}
where $v_0$ (respectively $v_1$) is solution to \eqref{eq:NonlinearSchrodinger_torus_nonlinearpart} with $\phi_0$ (respectively with $\phi_1$).
It follows from the last inequality, together with the trilinear estimates \eqref{eq:estimationtrilinear1}, that
\begin{align*}\|B\phi_0 -B\phi_1\|_{H^1(\T^d)} & \leq C \left(\|u_0\|^2_{X^{1,b'}_T}+\|u_1\|^2_{X^{1,b'}_T}\right) \|u_0- u_1\|_{X^{1, b'}_T}\\
& \leq C \left(\|u_0\|^2_{X^{1,b}_T}+\|u_1\|^2_{X^{1,b}_T}\right) \|u_0- u_1\|_{X^{1, b}_T}.
\end{align*}
By using once again the fact that the flow map given by \eqref{eq:flowmap} is Lipschitz on bounded set, we obtain a new constant $C>0$ independant on $\varepsilon_0$ such that
\begin{align*}\|B\phi_0- B\phi_1\|_{H^1(\T^d)} &\leq C \left(\|\phi_0\|^2_{H^1(\T^d)}+ \|\phi_1\|^2_{H^1(\T^d)} \right) \|\phi_0-\phi_1\|_{H^1(\T^d)}\\
& \leq 2C\varepsilon^2_0 \|\phi_0-\phi_1\|_{H^1(\T^d)}.
\end{align*}
By now, we set $\varepsilon_0 \leq \frac1{2\sqrt C}$ and $B : \overline{B_{H^1(\T^d)}}(0, \varepsilon_0) \longrightarrow \overline{B_{H^1(\T^d)}}(0, \varepsilon_0)$ is a contraction mapping and admits an unique fixed point, according to the Banach fixed point Theorem.
\end{proof}

\subsubsection{Third step: application of the stabilization result.}

According to Proposition~\ref{prop:local_control}, there exists $\varepsilon>0$ such that for all $w_0 \in H^1(\T^d)$ satisfying $\|w_0\|_{H^1(\T^d)} \leq \varepsilon$ there exists $g \in \mathcal C([0,1], H^1(\T^d))$ supported in $[0,1] \times \overline{\omega}$ so that the unique solution $u \in X^{1,b}_T$ of \eqref{eq:NonlinearSchrodinger_torusControllednear0} satisfies $w(1,\cdot)=0$.

Let $R\geq 1$ and $u_0 \in H^1(\T^d)$ such that $E(u_0) \leq R$. By the stabilization result of Theorem~\ref{thm:stabilisation_main_result}, there exists a control $h_1 \in \mathcal C([0,+\infty), H^1(\T^d))$ such that the solution of \eqref{eq:NonlinearSchrodinger_torusControlled} satisfies 
$$\forall t\geq 0, \quad E(u(t)) \leq Ce^{-\gamma t} E(u_0),$$
where $C, \gamma$ are positive constants only depending on $\omega$.
In particular, for $T= \gamma^{-1} \ln R + \gamma^{-1} \ln \left(\frac C{\varepsilon^2}\right)$, we have 
$$\|u(T)\|^2_{H^1(\T^d)} \leq \sqrt{E(u(T))} \leq  C e^{-\gamma T} R \leq \varepsilon^2.$$
On the other hand, thanks to Proposition~\ref{prop:local_control}, there exists a control $h_2 \in \mathcal C([0,1], H^1(\T^d))$ supported in $[0,1]\times \overline \omega$ such that the solution $\tilde u$ of \eqref{eq:NonlinearSchrodinger_torusControlled} started from $u(T)$ satisfies $\tilde u(1, \cdot)=0$. 

To conclude, it suffices to define the control by $h(t, \cdot)=h_1(t,\cdot)$ on $[0,T]$ and $h(t, \cdot) = h_2(t-T, \cdot)$ on $[T, T+1]$. With this choice of control, the solution $u$ of \eqref{eq:NonlinearSchrodinger_torusControlled} satisfies $u(T+1,\cdot)=0$ and $T+1$ is a controllability time independent on $u_0$. 
In particular, $\tau(R) \leq T+1 \leq \tilde C \ln (R+1)$
where $\tilde C$ is a positive constant depending on $\gamma$ and $C$.

\bibliographystyle{alpha}
\small{\bibliography{NLSchrodinger}}

\end{document}